\documentclass{article}
\usepackage{float}
\usepackage{graphicx}
\usepackage{tikz-cd}
\makeatletter
\usepackage{amsmath}
\usepackage{amssymb}
\usepackage{amsthm}
\usepackage{centernot}
\usepackage{mathrsfs}
\usepackage[margin=1.5in]{geometry}
\usepackage{comment}

\usepackage[shortlabels,inline]{enumitem}
\usepackage{hyperref}
\usepackage[capitalise]{cleveref}
\crefformat{equation}{(#2#1#3)}

\newtheorem*{rep@theorem}{\rep@title}
\newcommand{\newreptheorem}[2]{\newenvironment{rep#1}[1]{\def\rep@title{#2 \ref{##1}}\begin{rep@theorem}}{\end{rep@theorem}}}

\newtheorem{thm}{Theorem}[section]
\newtheorem*{thm*}{Theorem}
\newreptheorem{thm}{Theorem}
\newtheorem{lma}[thm]{Lemma}
\newreptheorem{lma}{Lemma}
\newtheorem*{lma*}{Lemma}
\newtheorem*{prp*}{Proposition}
\newtheorem{prp}[thm]{Proposition}
\newreptheorem{prp}{Proposition}
\newtheorem{cor}[thm]{Corollary}
\newreptheorem{cor}{Corollary}
\newtheorem*{cor*}{Corollary}
\newtheorem{conj}[thm]{Conjecture}
\newtheorem*{conj*}{Conjecture}
\newtheorem{clm}{Claim}
\newtheorem*{clm*}{Claim}

\newtheorem*{obs*}{Observation}

\theoremstyle{definition}
\newtheorem{defn}[thm]{Definition}
\newtheorem*{defn*}{Definition}
\newtheorem{qst}[thm]{Question}
\newtheorem*{qst*}{Question}

\newtheorem*{rmk*}{Remark}
\newtheorem*{problem*}{Problem}

\newcommand{\HOD}{\textnormal{HOD}}
\newcommand{\Hom}{\textnormal{Hom}}

\newcommand{\OD}{\textnormal{OD}}

\newcommand{\Ord}{\textnormal{Ord}}

\newcommand{\Reg}{\textnormal{Reg}}
\newcommand{\Ult}{\textnormal{Ult}}
\newcommand{\Col}{\textnormal{Col}}

\newcommand{\ZF}{\textnormal{ZF}}
\newcommand{\ZFC}{\textnormal{ZFC}}
\newcommand{\UA}{\textnormal{UA}}
\newcommand{\AD}{\textnormal{AD}}
\newcommand{\AC}{\textnormal{AC}}
\newcommand{\DC}{\textnormal{DC}}

\newcommand{\cf}{\textnormal{cf}}

\newcommand{\meas}{\textnormal{meas}}

\newcommand{\id}{[\textnormal{id}]}

\newcommand{\RK}{\leq_{\textnormal{RK}}}

\newcommand{\ke}{<_{\Bbbk}}

\newcommand{\mo}{\lhd}

\newcommand{\Pbd}{P_\textnormal{bd}}

\newcommand{\crit}{\textnormal{crit}}
\newcommand{\ot}{\textnormal{ot}}

\newcommand{\ran}{\textnormal{ran}}

\newcommand{\ulim}{\textnormal{-}\lim}

\newcommand{\Hull}{\textnormal{Hull}}

\makeatletter\@addtoreset{case}{thm}\makeatother
\makeatother
\usepackage{xcolor}
\usepackage{mathtools}
\usepackage{adjustbox}
\usepackage{booktabs,array}
\newcolumntype{C}[1]{>{\centering\arraybackslash}m{#1}}
\usepackage{subcaption}
\usepackage[font=small,labelfont=bf,labelsep=colon]{caption}

\crefname{lma}{Lemma}{Lemmas}
\crefname{cor}{Corollary}{Corollaries}
\crefname{defn}{Definition}{Definitions}
\crefname{conj}{Conjecture}{Conjectures}

\begin{document}
\title{Measures in Nairian models, I}
\author{Gabriel Goldberg\thanks{
    Gabriel Goldberg was partially supported by the National Science Foundation under Grant No. 2401789.
}\\ UC Berkeley
\and
Grigor Sargsyan\thanks{
    Grigor Sargsyan was supported by the National Science Centre, Poland, through MAESTRO 15, grant no. UMO-2023/50/A/ST1/00258.
}\\ IMPAN}
\maketitle
\begin{abstract}
    Assuming \(\AD^+ + V = L(P(\mathbb R))\), we 
    show that the Ultrapower Axiom holds in
    \(\HOD_s\) for every \(\omega\)-sequence of ordinals \(s\).
    We use this result to analyze the structure of
    Nairian models \cite{NairianModels,FailureSquare},
    showing that they contain many measurable cardinals. 
    As an application, we establish
    the consistency of the existence of infinitely many consecutive measurable cardinals with ZF, starting with a Woodin limit of Woodin cardinals. This solves a question of
    Henle--Mathias--Woodin \cite{HenleMathiasWoodin}.  
    From the same large cardinal hypothesis, adapting a forcing technique of Adkisson--Ben-Neria \cite{AdkissonBenNeria2025}, we produce a model of ZF in which every uncountable cardinal is singular. Gitik's original construction \cite{Gitik} requires a proper class of strongly compact cardinals.
\end{abstract}

\section{Introduction}
The Axiom of Determinacy (AD) produces an intricate
structure theory for sets of reals and relatively small uncountable cardinals.
Can this theory be generalized in some way to a global structure theory for all sets and all cardinals?
Recent work of Blue--Larson--Sargsyan \cite{NairianModels,FailureSquare} offers a precise way to approach this question:
under strong determinacy hypotheses,
a candidate global theory emerges in
\textit{Nairian models},
a family of inner models of ZF 
obtained, roughly speaking, by 
considering the sets that are definable from  
countable sequences of ordinals.

Blue--Larson--Sargsyan's results on Nairian models rely on
tools from descriptive inner model theory, particularly
the fine-structural analysis of the hereditarily ordinal definable sets (\(\HOD\)) in models of determinacy.
An obstacle to our current understanding of Nairian models is
that no general fine-structural analysis is known 
for \(\HOD_s\), the class of sets hereditarily ordinal definable from a fixed countable sequence of ordinals \(s\).

Building on work of Woodin, we develop an abstract comparison process for \(\HOD_s\),
which shows that this structure satisfies the Ultrapower Axiom (UA). This set-theoretic principle, introduced by the first author \cite{UABook}, captures part
of the combinatorial content of fine-structural comparison 
without reference to fine structure.
\begin{repthm}{thm:ua_hod_s} 
    Assume \(\AD^+ + V = L(P(\mathbb R))\).
    Then for all \(s\in \Ord^\omega\), \(\HOD_s\) satisfies \(\UA\).
\end{repthm}

A central conjecture of inner model theory is
that under \(\AD^+ + V = L(P(\mathbb R))\), \(\HOD\) is a fine-structural inner model.
Woodin confirmed a precise consequence of this conjecture by showing 
that under \(\AD^+ + V = L(P(\mathbb R))\), \(\HOD\) satisfies UA \cite[Theorem 3.1]{GoldbergGeneralizationsUA}.
We reorganize the proof of Woodin's result and extend it to \(\HOD_s\) for every countable sequence of ordinals \(s\).
The theorem hints that \(\HOD_s\) might admit a fine-structural analysis, but this remains unclear.

We use \cref{thm:ua_hod_s} to study measures in Nairian models.
The basic idea is to analyze the restrictions of these measures to \(\HOD_s\).
Instead of using fine structure in \(\HOD_s\), we use the theory of supercompact cardinals under UA \cite{UABook}.
An interesting feature of our approach is that this theory,
which was originally developed to predict the structure of inner models
with supercompact cardinals, finds concrete applications
to models that do not contain an inner model with a supercompact.

We state our theorems in the context of \[\AD^+ + V = L(P(\mathbb R))\]
As usual, \(|X|\leq |Y|\) means there is an injection from \(X\) to \(Y\) and \(|X|\leq^* |Y|\) means there is a surjection from \(Y\) to \(X\).
Let
\[\aleph(X) = \sup\{\alpha\in \Ord : |\alpha| \leq |X|\},\quad \Theta(X) = \sup\{\alpha\in \Ord : |\alpha| \leq^* |X|\}\]
and let \(\Theta = \Theta(\mathbb R)\).

Fix an ordinal \(\xi \leq \Theta\). 
The \textit{extended Nairian model at \(\xi\)} is the class 
\(N_\xi^+\) of all sets hereditarily ordinal
definable from a bounded countable subset of \(\xi\).
Thus,
\[N_\xi^+ = \HOD_{I_\xi}\]
where \(I_\xi =  \textstyle\bigcup_{\beta < \xi}\beta^\omega\). 
When \(\xi\) is fixed, we will denote \(N_\xi^+\) by \(N^+\). 

We are interested in the situation where \(\xi\) is strongly inaccessible in 
\(N^+_\xi\). 
Blue--Larson--Sargsyan show that this is consistent relative to a Woodin limit of Woodin cardinals \cite[Theorems 8.1, 10.6, and 10.8 and Corollary 10.17]{NairianModels}.
In this case, the \textit{Nairian model at \(\xi\)} is given by:
\[N_\xi = N^+_\xi\cap V_\xi\]
This is a model of ZF exhibiting global determinacy phenomena. 
When \(\xi\) is fixed, we denote \(N_\xi\) by \(N\).

\textit{From now on we fix a \(\text{\(\xi\leq \Theta\) that is strongly inaccessible in \(N^+ = N^+_\xi\)}\).}
\begin{repthm}{thm:main}
    Let \(\mathcal H = \HOD\cap V_\xi\). Inside \(N\), suppose \(\eta\) is a regular cardinal satisfying: 
    \begin{enumerate}[(1)]
        \item
        For all \(\alpha < \eta\), \(P(\alpha) \subseteq \mathcal H({\eta^\omega})\). 
        \item  For all \(\alpha < \eta\), \(\Theta(\alpha^\omega)\leq \eta\).
    \end{enumerate}
    Then in \(N\), either \(\eta\) is measurable
    or \(\mathcal H\) correctly computes all finite successors of \(\eta^{++}\).
\end{repthm}
In the known Nairian models,
\(\mathcal H\) does not correctly compute a single successor cardinal below \(\xi\). Letting \(\lambda = (\eta^{+\omega})^{N}\), correct computation of successors in particular implies that some ordinal \(\kappa\leq \eta^{++}\) is \({<}\lambda\)-supercompact in \(\HOD\), whereas Blue--Larson--Sargsyan always assume Steel's \textit{HOD pair capturing} hypothesis (HPC) \cite{Steel2022Comparison},
which prevents \(\HOD\) from containing any nontrivial levels of supercompactness.

For our applications of \cref{thm:main}
we adopt Blue--Larson--Sargsyan's hypotheses:
\begin{equation}\label{eqn:BLS-hypotheses}\AD_\mathbb R + \Theta\text{ is regular} + V = L(P(\mathbb R)) + \textnormal{HPC}\end{equation}
We also assume that \(N\) is a minimal Nairian model
in the sense of \cite[Definition 1.1]{FailureSquare}.
We apply \cref{thm:main} in \(N\) to the sequence of cardinals
\[\langle \Theta(V_\alpha) : \omega \leq \alpha < \xi\rangle\]
These are the choiceless analogs of the Beth numbers. 
These cardinals are closely related to the strong cardinals of \(\HOD\); see \cref{prp:strongs}.
Note that
\(\Theta(V_\omega) =\omega_1\) and \(\Theta(V_{\omega+1}) = \Theta\).
It is a theorem of ZF \cite[Lemma 2.1]{GoldbergCCCP} that if \(\alpha\) is a limit ordinal, then \[\Theta(V_\alpha) = (\sup_{\beta < \alpha} \Theta(V_\beta))^+\] 
and a similar argument shows that under
\(\AD\), \(\Theta(V_{\omega+2}) = \Theta^+\) \cite[Proposition 2.2]{GoldbergCCCP}.

Blue--Larson--Sargsyan show that in \(N\), for \(n < \omega\),
\begin{equation}\Theta(V_{\omega+n+1}) = \Theta^{+n}\label{eqn:bls_eqn}\end{equation}
(See \cite[Theorem 1.3(1)--(3)]{FailureSquare} and
\cref{lma:V_vs_I}.)
More generally, although the value of \(\Theta(V_{\alpha+1})\) for limit ordinals \(\alpha\) of countable cofinality remains open,
for all other infinite \(\alpha\), 
Blue--Larson--Sargsyan show \cite[Theorem 10.23]{NairianModels}
\[\Theta(V_{\alpha+1}) = \Theta(V_{\alpha})^+\]
This can be seen as a version of the generalized
continuum hypothesis: in ZFC, GCH holds if and only if
\(\Theta(V_{\alpha+1}) = \Theta(V_\alpha)^+\) for all
infinite ordinals \(\alpha\).

Our main measurability result for minimal Nairian models is as follows:
\begin{repthm}{thm:theta_meas}
    In \(N\), for all infinite ordinals \(\alpha\),
    \(\Theta(V_\alpha)\) is measurable.
\end{repthm}
Applying Blue--Larson--Sargsyan's calculation \cref{eqn:bls_eqn}, we have our first application:
\begin{repcor}{cor:consistency}
    The theory ``\textnormal{ZF + }there are
    infinitely many consecutive measurable cardinals'' is consistent relative to a Woodin limit of Woodin cardinals.
\end{repcor}

\cref{thm:theta_meas} also allows us to compute \(\aleph(V_{\alpha+1})\).
\begin{repcor}{cor:no-long-rank-sequences}
    In \(N\), for all infinite ordinals \(\alpha\), \(\aleph(V_{\alpha+1}) = \Theta(V_\alpha)\).
\end{repcor}
In other words,
there is no \(\Theta(V_\alpha)\)-sequence of distinct subsets of \(V_\alpha\).
This is further evidence
that the theory of Nairian models yields a 
global generalization of determinacy.
The equation \(\aleph(V_{\omega+1}) = \Theta(V_\omega)\) asserts that there is no \(\omega_1\)-sequence of distinct reals, while \(\aleph(V_{\omega+2}) = \Theta(V_{\omega+1})\) states that there is no \(\Theta\)-sequence of distinct subsets of \(\mathbb R\).
These are well-known consequences of \(\AD\) and \(\AD_\mathbb R\) respectively.

Our second application is a reduction of the
large cardinal hypothesis used in Gitik's construction of
a model of ZF in which every uncountable
cardinal is singular \cite{Gitik}. This uses our results
on the global structure of measures in arbitrary Nairian models.
We return to the setting of \(\AD_\mathbb R + \DC + V = L(P(\mathbb R))\), with \(\xi < \Theta\) inaccessible in \(N^+\).

\cref{thm:unique_complete} states that
in \(N\), 
each regular cardinal \(\delta\) carries a unique weakly 
normal measure concentrating on ordinals of countable cofinality,
which we denote by \(W_\delta\). \cref{thm:monotonicity}
shows that in \(N\), 
the map that sends \(\delta\) to the completeness \(\kappa_\delta\) of \(W_\delta\) is monotone.
\cref{prp:hpc_no_sc} and \cref{prp:cardinal_no_sc}
show that if either HPC holds or \(\xi\) is a cardinal, 
then the cardinals \(\kappa_\delta\) are unbounded in \(N\). 
In \cref{thm:abn}, assuming this unboundedness, we adapt 
a forcing technique of Adkisson--Ben-Neria \cite{AdkissonBenNeria2025}, 
originally developed in the context of choiceless large cardinals,
to work over \(N\).
This produces a model of ZF in which
every uncountable cardinal is singular,
leading to our second application:
\begin{repcor}{cor:gitik}
    The theory ``\textnormal{ZF + }every uncountable cardinal is singular'' 
    is consistent relative to a Woodin limit of Woodin cardinals.\qed
\end{repcor}

Finally, we develop a comparison process 
for Nairian models themselves. We assume
\(\AD_\mathbb R + \DC + V = L(P(\mathbb R))\)
and \(\xi < \Theta\). 
By comparing Nairian models, 
we prove the following results
for \(N\):
\begin{repthm}{thm:unique}
    In \(N\), for every set \(X\), there is a unique normal fine measure
    on \(P_{\omega_1}(X)\).
\end{repthm}

\begin{repthm}{thm:ua_N}
    \textnormal{UA} holds in \(N\).
\end{repthm}
In this choiceless context, UA states that 
any two ultrapowers of \(V\) 
by measures on ordinals can be compared by further 
internal ultrapowers. Note that these ultrapowers will not in general be fully elementary.

As an application of \cref{thm:ua_N}, we prove the following result. We return to the Blue--Larson--Sargsyan hypotheses \cref{eqn:BLS-hypotheses} and assume \(N\) is a minimal Nairian model. 
\begin{repthm}{thm:filter_extension}
    In \(N\), suppose \(\alpha\geq \omega\) is not of the form \(\beta+1\) where \(\beta > \omega\) and \(\cf(\beta) = \omega\).
    Let \(\delta = \Theta(V_\alpha)\). Then every
    \(\delta\)-complete filter on an ordinal extends to a
    \(\delta\)-complete ultrafilter.
\end{repthm}

\paragraph{Historical notes.}
Blue and Sargsyan realized in 2025 that since the Nairian model
contains many runs of consecutive regular cardinals \cite{NairianModels}, 
one could answer the Henle--Mathias--Woodin consecutive measurables question by proving that all regular cardinals of the Nairian model are measurable. Moreover,
they realized that the latter would follow from a directed system analysis of models of the form \(\HOD_s\), for \(s\) a countable sequence of ordinals, analogous to the known analysis of \(\HOD\) in models of determinacy.
Later Trang realized one could settle the consecutive measurables question assuming Sargsyan's quasi-club conjecture; see the
comments following \cref{thm:unique}.
Both the \(\HOD_s\)-analysis and the quasi-club conjecture remain open.  
\section{Definability from a countable sequence of ordinals}\label{sec:hod_of_s}
    The main result of this section is an abstract comparison
    theorem for models of the form \(\HOD_s\), where \(s\) is a countable sequence of ordinals. 
    Assume \(\AD^+ + V = L(P(\mathbb R))\). 
    We will show that for all \(s\in \Ord^\omega\), \(\HOD_s\vDash \UA\) and the \(\HOD_s\)-ultrafilter hypothesis holds:

    \begin{defn}Suppose \(M\) is a transitive model of set theory. The  \textit{\(M\)-ultrafilter hypothesis}
    states that if \(U\in M\) is a measure in \(M\),
    \(U\) extends to a countably complete filter in \(V\).\end{defn} 
    Here and throughout, a \textit{measure} is a countably complete ultrafilter.

    The proof is obtained by relativizing Woodin's theorem that under \(\AD^+ + V = L(P(\mathbb R))\), \(\HOD\vDash \UA\) and the \(\HOD\)-ultrafilter hypothesis holds. While Woodin's original proof easily relativizes to a real, relativization to a countable sequence of ordinals is not immediate, and Woodin explicitly left this question open.
    Our approach uses the following ZF fact, essentially due to Paris \cite{Enayat, Paris}:

    \begin{defn}
    Let \(\mathscr T_X\) denote the \(\Sigma_2\)-theory of \(V\) with parameters in
    \(\Ord\cup \{X\}\). 
    Let \(\mathscr H_X\) denote the structure \((\HOD_X,\mathscr T_X,{\in})\).
    Also let \(\mathscr T = \mathscr T_\emptyset\) and \(\mathscr H = \mathscr H_\emptyset\).
    \end{defn}

    \begin{prp}\label{prp:hod-realization}
    Suppose \(X\) is a set, \(H\) is a countable transitive structure, and \[\rho : H\to \mathscr H_X\] is an elementary embedding. Then there is a countable transitive model \(\bar V\) of \(\ZF\)
    and a set \(\bar X\in \bar V\)
    such that \(H = (\mathscr H_{\bar X})^{\bar V}\).
    Thus \(\rho\) is partial elementary from \((\bar V,\bar X)\) to \((V,X)\).
    \end{prp}

We begin by providing a new proof of Woodin's original theorem. Our proof reorganizes, simplifies, and generalizes a proof that Woodin gave in the case \(\AD_\mathbb R + \Theta\) is regular; see \cite{GoldbergGeneralizationsUA}. We use one of the central \(\AD^+\) theorems on derived models:
\begin{thm}[{Woodin, \cite[Theorem 2.33]{SteelOrdinalDefinability}}]\label{thm:derived}
    Assume \(\AD^+ + V = L(P(\mathbb R))\). 
    There is a definable inner model \(M\subseteq \HOD\) with a limit of Woodin cardinals \(\eta\)
    and a generic elementary embedding \[j : V\to D\]
    where \(D\) is a derived model of \(M\) at \(\eta\).
\end{thm}
Here we mean that \(M\) is definable without parameters over \(V\),
or equivalently over \(\mathscr H\). 

We can make a number of further conclusions abstractly from the statement of \cref{thm:derived}, which are also immediate consequences of Woodin's unpublished construction. By homogeneity, \(\mathscr H^D\) is definable over \(M\). Since \(\HOD = L[S]\) for a set of ordinals \(S\), \(j\restriction \mathscr H\) is an extender embedding.
By the uniqueness of generic extender embeddings,
\(j\restriction \mathscr H\) is definable without parameters over \(V\), or equivalently over \(\mathscr H\).
A basic fact about derived models is that \(\Theta^D \leq (\eta^+)^M\).
It follows easily that \(\eta \geq \Theta\) and that \(\eta\) is the supremum of the measurable cardinals of \(M\);
it follows that \(\eta\) is the supremum of the Woodin cardinals of \(M\).
\begin{thm}[Woodin]\label{thm:ua}
    Assume \(\AD^+ + V = L(\mathbb R,\Hom)\). Then
    \(\mathscr H\) satisfies weak comparison:
    if \(H_0\) and \(H_1\) are finitely generated
    structures 
    elementarily embeddable into \(\mathscr H\),
    then either \(H_0\in H_1\), \(H_1\in H_0\),
    or there are close embeddings
    \(k_0 : H_0 \to H^*\) and \(k_1 : H_1\to H^*\).
\end{thm}
\begin{proof}
    Let \(M\) and \(\eta\)
    be as in \cref{thm:derived}, and let \(\mathbb R^* = \mathbb R \cap M\).

    Suppose that \(H\in \mathscr H\) is a countable transitive structure that
    is elementarily embeddable into \(\mathscr H\).\footnote{
        Throughout this section, we work in choiceless Morse--Kelley set theory, 
        which allows us to refer to the first-order satisfaction predicate of \(V\).
        With a bit more work, the proofs can be carried out in \(\ZF\).
    } 
    We begin by constructing a 
    close elementary embedding
    \(k : H \to H^*\) where \(H^*\) is \(\mathscr H^{D^*}\) 
    for a model \(D^*\) of the form \(L_\alpha(\mathbb R^*,\Gamma)\)
    with \(\Gamma\) an initial segment of \(\Hom_{<\eta}^{M}\).

    By \cref{prp:hod-realization}, choose a countable transitive model
    \(\bar V\equiv V\) such that \(\mathscr H^{\bar V}=H\).
    Applying \cref{thm:derived} inside \(\bar V\), let
    \(\bar M\subseteq H\), \(\bar\eta\), and the generic elementary embedding
    \(j : \bar V\to D\) be as in that theorem, where
    \(D\) is a derived model of \(\bar M\) at \(\bar\eta\).
    Define \[i : H\to\mathscr H^D\]
    by \(i = j\restriction H\).
    By the observations following \cref{thm:derived}, \(i\) is definable over \(H\).

    Note that \(\bar M\in M\) and by absoluteness,
    \(M\) sees that \(\bar M\) is elementarily embeddable into \(M\).
    Therefore using Neeman-style \(\mathbb R\)-genericity iterations,\footnote{
        Woodin's original construction used Woodin-style genericity iterations,
        which requires establishing certain instances of the unique branch hypothesis.
        Steel suggested replacing them with Neeman-style iterations using \cite{Steel}.
    } we can iterate \((\bar M,\bar\eta)\) 
    to a model \((M^*,\eta^*)\), again elementarily embeddable into \(M\),
    such that \(\mathbb R^* = \mathbb R\cap M\) are the reals of a symmetric collapse of \(M^*\) at
    \(\eta^*\), and if \(D^*\) is the derived model of \(M^*\) at \(\eta^*\) in \({M}^*(\mathbb R^*)\), then
    \(\Gamma = \Hom^{D^*}\) is an initial segment of 
    \(\Hom_{<\eta}^{M}\).

    Let \[\pi : \bar M \to {M}^*\] be the iteration map.
    Then \(\pi\restriction {\mathscr{H}}^D\) is elementary from
    \({\mathscr{H}}^D\) to \({\mathscr{H}}^{D^*}\).
    Set \(H^* = {\mathscr{H}}^{D^*}\). Define an elementary embedding \(k : H \to H^*\)
    by \[k = \pi \circ i\] 

    We claim that \(k\) is close to \(H\). 
    To prove this, fix \(A\in H^*\), and let us show that
    \(k^{-1}[A] \in H\). We have 
    \(\pi^{-1}[A]\in \bar M\) since \(\pi\) is close to \(\bar M\) and \(A\in M^*\).
    Since \(\bar M\subseteq H\), it follows that \(\pi^{-1}[A]\in H\) and hence \(i^{-1}[\pi^{-1}[A]]\in H\) since \(i\) is definable over \(H\). But \(k^{-1}[A] = i^{-1}[\pi^{-1}[A]]\).
    
    We will use that \[D^* = \Hull^{D^*}(k[H]\cup\Lambda)\]
    where \(\Lambda = P(\mathbb R)\cap D^*\).
    This follows from the fact that \(H^* = \Hull^{H^*}(k[H]\cup k(\Theta^H))\),
    which holds because \(k\) is close to \(H\)
    and every uniform measure of \(H\) lies on a cardinal less than
    or equal to \(\Theta^H\).

    We now show that Weak Comparison holds for \({\mathscr{H}}\).
    Fix any two
    finitely generated models \(H_0\) and \(H_1\) that are elementarily embeddable into \({\mathscr{H}}\). We will construct close embeddings \(k_0 : H_0 \to H^*_0\) and \(k_1 : H_1\to H^*_1\)
    such that \(H^*_0\) and \(H_1^*\) are \textit{comparable}: either they are equal or one belongs to the other.

    For \(\ell = 0,1\), the construction above yields
    a close embedding \(k_\ell : H_\ell\to H_\ell^*\) where \(H_\ell^* =  {\mathscr{H}}^{D_\ell^*}\) for
    a model \(D_\ell^*\) satisfying 
    \(\mathbb R\cap D_\ell^* = \mathbb R\cap M\), \(\Gamma_\ell = \Hom^{D_\ell^*}\mo \Hom_{<\eta}^{M}\),
    and \(D_\ell^* = \Hull^{D_\ell^*}(k_\ell[H_\ell]\cup \Lambda_\ell)\)
    where \(\Lambda_\ell = P(\mathbb R^*)\cap D_\ell^*\).

    We claim that \(D_0^*\) and \(D_1^*\) are comparable. 
    By Wadge determinacy in \(\Hom_{<\eta}^{M}\), \(\Gamma_0\) and \(\Gamma_1\) are comparable pointclasses. It follows from our assumption that \(V = L(\mathbb R,\Hom)\)
    that \(\Lambda_0\) and \(\Lambda_1\) are comparable as well.
    (If \(\Gamma_0 = \Gamma_1\), this is trivial, and if \(\Gamma_0 \mo \Gamma_1\),
    then \(\Gamma_0^\#\in D_1^*\), so \(\Lambda_0\mo \Lambda_1\).)
    By symmetry, it suffices to show that if \(\Lambda_0 \mo \Lambda_1\), then \(D_0^*\in D_1^*\). 

    Suppose first that every set of reals has a sharp. 
    We will use that \(D_1^*\) is correct about
    sharps, which means that if \(A\in D_1^*\) and \(D_1^*\vDash B = A^\#\), then 
    in fact, \(B = A^\#\). 
    To see this, note that \(M_1^*\) is correct about sharps
    because \(M_1^*\) is elementarily embeddable into \(M\).
    It follows that \(D_1^*\) is correct about
    sharps, since  \(D_1^*\) is an inner model of a forcing extension of \(M_1^*\).

    Since \(D_{0}^* = \Hull^{D_0^*}(k_0[H_0]\cup\Lambda_0)\) and \(H_0\) is finitely generated, \(D_0^*\) is coded by a subset of \(\Lambda_0\) in \(L(\mathbb R^*,\Lambda_0)\). 
    Since \(\Lambda_0^\#\in D_1^*\), \(D_0^*\in D_1^*\). 

    Suppose instead that some set of reals has no sharp. 
    We then have that \(\theta = \Theta^{H_0}\) is regular in \(H_0\). It follows that
    \(k_0(\theta) = \sup k_0[\theta]\) since \(\theta\) carries no uniform measures in \(H_0\) and \(k_0\) is close to \(H_0\). 
    In particular, \(\cf^{M}(k_0(\theta)) = \omega\).
    (We use that \(k_0\in M\), which follows from the fact that \(H_0^*\in M\) 
    and \(H_0\) is finitely
    generated.) Therefore \(D_{0}^* = \Hull^{D_0^*}(\mathbb R^*\cup \mathcal A)\) where 
    \(\mathcal A = \{a_n : n < \omega\}\) 
    is a countable set in \(M\). Letting \(T_n = \text{Th}^{D_0^*}_{\Sigma_n}(\mathbb R^*\cup \{a_0,\dots,a_{n-1}\})\), we have \(\{T_n\}_{n < \omega}\subseteq D_0^*\) and \(\langle T_n\rangle_{n < \omega}\in M\), and hence \(\langle T_n\rangle_{n < \omega}\in D_1^*\).
    Since \(D_0^*\) can be recovered from  \(\langle T_n\rangle_{n < \omega}\), it follows that
    \(D_0^* \in D_1^*\).

    Thus we have shown that \(D_0^*\) and \(D_1^*\) are comparable.
    If \(D_0^*\in D_1^*\), for example, then \(D_0^*\) is \(\OD\) in \(D_1^*\) and hence
    \(H_0^*\in H_1^*\). It follows that \(H_0^*\) and \(H_1^*\) are comparable, and
    this establishes weak comparison. (Note here that if \(H_0^*\in H_1^*\), then \(H_0\in H_1\).)
\end{proof}
The conclusion of \cref{thm:ua} 
implies \(\HOD\) satisfies \(\UA\) by \cite[Theorem 2.3.10]{UABook}, and a further reflection argument
allows us to make this conclusion
assuming \(V = L(P(\mathbb R))\) rather than
\(V = L(\mathbb R,\Hom)\).
We omit this reflection argument, which is similar to
the one in \cref{thm:ua_hod_s} below.
It is open whether \(\mathscr H\) satisfies weak comparison assuming only \(\AD^+ + V = L(P(\mathbb R))\).

We now prove our main relativizations of Woodin's result. The proof will again require a reduction to the case that \(V  = L(\mathbb R,\Hom)\).
This uses Woodin's analysis of the largest Suslin cardinal:
\begin{thm}[Woodin]\label{thm:woodin_strong}
    Assume \(\AD^+ + V = L(P(\mathbb R))\) and \(\delta\) is the largest Suslin. If
    \(\lambda < \Theta\),
    there is a normal measure \(U\) on \(\delta\)
    with \(j_U(\delta) > \lambda\) and
    \(P(\lambda)\cap j_U(\HOD) = P(\lambda)\cap \HOD\).\qed
\end{thm}
The proof of this theorem is unpublished in the key case \(V \neq L(\mathbb R,\Hom)\).

In the theorem below, as well as in  
\cref{thm:unique} and \cref{thm:ua_N},
we make use of a technique for recovering \L o\'s's
theorem in the context of ZF due to Mitchell Spector \cite{Spector1988,Spector1991}.
\begin{thm}[Spector]\label{thm:spector}
    Assume \(\ZF + \DC + V = L(X)\) for some set \(X\). 
    Then for any measure \(U\), in a forcing extension,
    there is an inner model \(M\), an elementary embedding \[j : V \to M\]
    and an embedding \(k : \Ult(V,U)\to M\)
    such that \(\Ord\subseteq \ran(k)\) and
    \(k\circ j_U = j\).
\end{thm}
See \cite[Corollary 1.3, Theorems 3(b), 4, and 8]{Spector1991}.
For the purposes of \cref{thm:ua_hod_s}, one could use Woodin's generic ultrapowers
as in \cite[Theorem 4.7]{UAFromDet}. In fact, Woodin's construction is equivalent to Spector's
in the context of measures on ordinals below \(\Theta\) under \(\AD\).
On the other hand, Spector's more general construction seems
essential to the later results, \cref{thm:unique} and \cref{thm:ua_N}.

\begin{thm}\label{thm:ua_hod_s}
    Under \(\AD^+ + V = L(P(\mathbb R))\), 
    \(\HOD_s\) satisfies \(\UA\) for every \(s \in\Ord^\omega\).
    \begin{proof}
        We first prove the theorem in the case that \(V = L(\mathbb R,\Hom)\).
        Let \(H\) be a pointwise definable transitive structure that is elementarily embeddable into \(\mathscr H\). 
        
        Fix a countable transitive model \(\bar V\equiv V\) in \(\HOD\), constructed by type omission, so that
        \(H = \mathscr H^{\bar V}\). 
        Fix \(s\in \Ord^\omega\cap \bar V\). We will show \(\UA\) holds in \(P = \mathscr H_s^{\bar V}\).

        Fix internal ultrapowers \(e_0 : P\to P_0\) and \(e_1 : P \to P_1\), and let \(s_n = e_n(s)\). For \(n = 0,1\), we imitate the construction from \cref{thm:ua}. That is, construct \(\bar V_n\equiv V\) in \(\HOD\) by relativizing \cref{prp:hod-realization} 
        so that
        \[P_n = \mathscr H_{s_n}^{\bar V_n}\]
        Applying \cref{thm:derived} in \(\bar V_n\), fix \(\bar M_n\subseteq \mathscr H^{\bar V_n}\), \(\bar\eta_n\), and an elementary \(j_n : \bar V_n\to D_n\) where \(D_n\) is a derived model of \(\bar M_n\) at \(\bar\eta_n\).
        Let \(t_n =j_n(s_n)\) and define
        \[
            i_n : P_n\to\mathscr H^{D_n}_{t_n}
        \]
        by \(i_n = j_n\restriction P_n\).

        By \cref{thm:derived}, fix \(\gamma_n < \bar\eta_n\) such that \(t_n\) is
        \({<}\gamma_n\)-generic over \(\bar M_n\).
        Let \(\pi_n : \bar M_n\to M_n^*\) be a Neeman-style \(\mathbb R^*\)-genericity iteration above \(\gamma_n\), where \(\mathbb R^* = \mathbb R\cap \HOD\), such that if \(D_n^*\) is the derived model of \(M_n^*\) as computed in \(M_n^*(\mathbb R^*)\),
        then \(D_0^*\) and \(D_1^*\) are comparable.
        Let \(s_n^* = \pi_n[t_n]\), and let
        \[
            \pi_n^+ : \bar M_n[t_n]\to M_n^*[s_n^*]
        \]
        be the unique extension coming from L\'evy--Solovay.
        Finally let \(k_n = \pi_n^+\circ i_n\), so
        \[
            k_n : P_n\to \mathscr H_{s_n^*}^{D_n^*}
        \]
        is a close embedding.

        We first show that \(D_0^* = D_1^*\).
        We have elementary embeddings
        \begin{align*}
            (k_0\circ e_0)\restriction H : H\to \mathscr H^{D_0^*}\\ 
            (k_1\circ e_1)\restriction H : H \to \mathscr H^{D_1^*}
        \end{align*}
        Assuming towards a contradiction that 
        \(D_0^* \in D_1^*\), for example, we would have
        that \(H\) is coded by a real in \(\mathscr H^{D_1^*}\).
        Since \(\crit(k_1\circ e_1) > (2^{\aleph_0})^H\),
        \(H\) and \(\mathscr H^{D_1^*}\) have the same reals,
        and so \(H\) is coded by a real in \(H\), which is a contradiction. This proves that \(D_0^* = D_1^*\),
        so let us denote this model by \(D^*\).

        Since \(H\) is pointwise definable, it follows that
        \(k_0\circ e_0 \restriction H = k_1\circ e_1\restriction H\).
        Therefore
        \(k_0\circ e_0[s] = k_1\circ e_1[s]\),
        and hence \(k_0\circ e_0(s) = k_1\circ e_1(s)\).
        In other words, \(k_0(s_0) = k_1(s_1)\), and so
        \(\mathscr H_{k_0(s_0)}^{D^*} = \mathscr H_{k_1(s_1)}^{D^*}\).
        Now \(k_0\circ e_0\) and \(k_1\circ e_1\) have the same target model, and
        since every element of \(P = \mathscr H_s^{\bar V}\) is definable in \(P\)
        from \(s\), 
        and \(k_0\circ e_0(s) = k_1\circ e_1(s)\), 
        we obtain \(k_0\circ e_0 = k_1\circ e_1\). This yields \(\UA\) for \(e_0\) and \(e_1\), as desired.

        Now assume \(V \neq L(\mathbb R,\Hom)\).
        Let \(\delta\) be the largest Suslin cardinal. 
        We will show that 
        for all \(s\in \Theta^\omega\), \(\HOD_s\) satisfies \(\UA\). This suffices for the result: if \(t\in \Ord^\omega\), then since \(\HOD_t\) is a \(\Theta\)-cc extension of \(\HOD\), there is a set \(S\in \HOD\) such that \(\ot(S) < \Theta\) and \(\ran(t)\subseteq S\), and therefore \(\HOD_t = \HOD_s\) for some \(s\in \Theta^\omega\).

        Let \(\delta\) be the largest Suslin cardinal, which exists since
        \(\Hom\neq P(\mathbb R)\). 
        By Suslin-co-Suslin reflection, \(\Hom = \Hom^{L(\mathbb R,\Hom)}\)
        and for all \(s\in \delta^\omega\), \(\HOD_s\cap V_\delta = \HOD^{L(\mathbb R, \Hom)}_s\cap V_\delta \).
        Thus, \(L(\mathbb R,\Hom)\) satisfies \(\AD^+ + V = L(\mathbb R,\Hom)\),
        so by the previous case, \(\UA\) holds in \(\HOD^{L(\mathbb R, \Hom)}_s\) for all \(s\in \delta^\omega\).
        In the context of \(\ZFC\), since \(\UA\) is a \(\Pi_2\)-sentence, 
        \(\UA\) is downwards absolute to any 
        \(V_\rho\) with \(|V_\rho| = \rho\).
        Thus for all \(s\in \delta^\omega\),
        \(\UA\) holds in \[\HOD^{L(\mathbb R, \Hom)}_s\cap V_\delta = \HOD_s\cap V_\delta\]

        Assume towards a contradiction that \(\UA\) fails in 
        \(\HOD_s\) for some \(s\in\Theta^\omega\). 
        Fix a limit ordinal \(\gamma < \Theta\) such that \(\sup(s) < \gamma\) and \(\UA\) fails in \(\HOD_s\cap V_\gamma\).
        We claim there is some \(\lambda < \Theta\) large enough that
        \(\HOD_s\cap V_\gamma \subseteq (\HOD\cap V_\lambda)[s]\).
        For this, note that every bounded subset of \(\Theta\) in \(\HOD_s\) has a \(\mathbb B\)-name in \(\HOD\cap V_\Theta\) where \(\mathbb B\) is the \(\infty\)-Borel Vop\v{e}nka algebra adding an element of \(\gamma^\omega\) to \(\HOD\). 
        This is by a nice name argument, using that 
        \(\mathbb B\) is \(\Theta\)-cc in \(\HOD\)
        and \(\mathbb B\subseteq \HOD\cap V_\Theta\). 
        Now let \(S\) be a bounded subset of \(\Theta\) coding \(\HOD_s\cap V_\gamma\), and fix \(\lambda < \Theta\) large enough that this set has a \(\mathbb B\)-name \(\dot S\) in \(\HOD\cap V_\lambda\). Then
        \(\HOD_s\cap V_\gamma \subseteq (\HOD\cap V_\lambda)[s]\),
        since \(\dot S\) can be interpreted  
        in \((\HOD\cap V_\lambda)[s]\).
       
        Let \(j : V\to M\) be a Spector ultrapower by a \(\delta\)-complete measure on \(\delta\) such that
        \(j(\delta) > \lambda\), \(\HOD^M\cap V_\lambda = \HOD\cap V_\lambda\). Note that \(\Ord^\omega\cap V\subseteq M\) by countable choice, and in particular \(s\in M\). 
        Since all measures on ordinals are \(\OD\),
        \(\HOD^M_s \subseteq \HOD_s\). On the other hand,
        \(\HOD\cap V_\lambda\subseteq \HOD^M_s\), so \[\HOD_s\cap V_\gamma\subseteq (\HOD\cap V_\lambda)[s]\subseteq \HOD^M_s\] 
        It follows that \(\HOD_s\cap V_\gamma = \HOD^M_s\cap V_\gamma\). Therefore \(\UA\) fails in \(\HOD^M_s\cap V_\gamma\), and hence it fails in \(\HOD^M_s\cap V_{j(\delta)}\),
        contradicting that in \(V\),
        \(\UA\) holds in \(\HOD_s\cap V_\delta\)
        for any \(s\in \delta^\omega\).
    \end{proof}
\end{thm}

Using \cref{thm:derived}, we can generalize Woodin's theorem on the \(\HOD\)-ultrafilter hypothesis
to the context of \(\AD^+ + V = L(P(\mathbb R))\). As in the case of
\cref{thm:ua}, Woodin proved the slightly
weaker \((\HOD\cap V_\Theta)\)-ultrafilter hypothesis, again using more
\(\AD^+\)-theory.

The following theorem is due to Woodin when
\(s\in \omega^\omega\).
\begin{thm}\label{thm:hod_uf}
    Assume \(\AD^+ + V = L(P(\mathbb R))\). Then for all \(s\in \Ord^\omega\), the \(\HOD_s\)-ultrafilter hypothesis holds.
\end{thm}
\begin{proof}
    We can reduce to the case that \(V = L(\mathbb R,\Hom)\)
    by essentially the same argument from \cref{thm:ua_hod_s}.

    Suppose \(H\) is the pointwise definable hull of \(\mathscr H\).
    By \cref{prp:hod-realization}, we can form a model \(\bar V\equiv V\)
    such that \(H = \mathscr H^{\bar V}\). 
    Fix \(s\in \Ord^\omega\cap \bar V\), and suppose \(e : H_s \to H_1\) is the ultrapower of \(H_s = (\mathscr H_s)^{\bar V}\) by a measure \(U\in H_s\).
    It suffices to show that
    \(U\) extends to a countably complete filter in \(\bar V\).
    Fix \(\sigma \in (P_{\omega_1}(U))^{\bar V}\). We will show that
    \(\bigcap \sigma\neq \emptyset\).

    We use the construction from the proof of \cref{thm:ua_hod_s},
    with \(M\) and \(\eta\) as in \cref{thm:derived} and
    \(\mathbb R^*=\mathbb R\cap M\).
    Applying \cref{thm:derived} inside \(\bar V\), let
    \(\bar M\subseteq H\), \(\bar\eta\), and the generic elementary embedding
    \(j:\bar V\to D\) be as in that theorem, where
    \(D\) is a derived model of \(\bar M\) at \(\bar\eta\). 
    Since \(j(\sigma)\) and \(j(s)\) are in \(D\), they belong to a common
    \({<}\bar\eta\)-generic extension of \(\bar M\).
    Thus by taking the critical point of the genericity iteration
    \(\pi : \bar M \to M^*\) sufficiently large, we can ensure that
    \(\pi\) lifts uniquely to an elementary embedding
    \(\pi^+\) of
    \(\bar M[j(\sigma),j(s)]\).
    As in \cref{thm:ua_hod_s}, let \(k = \pi^+\circ j\restriction H_s\).

    Let \(Q = \mathscr H_{k(s)}^{D^*}\) be the target model of \(k\).
    The construction from \cref{thm:ua_hod_s} yields \(k_1 : H_1\to Q\) such that
    \(k = k_1\circ e\).
    
    Let \(a = \id_U\). Then \(a\in \bigcap e[\sigma]\),
    so \(k_1(a)\in \bigcap (k_1\circ e)[\sigma] =  \bigcap k[\sigma]\). 
    In particular, \(\bigcap k[\sigma]\neq \emptyset\), and 
    since \(k[\sigma] = \pi^+(j[\sigma])\), it follows that
    \(\bigcap j[\sigma]\neq \emptyset\).
    Finally, the elementarity of
    \(j:\bar V\to D\) implies \(\bigcap\sigma\neq\emptyset\).
\end{proof}

\section{Consequences for the Nairian model}\label{sec:soft_con}
In this section, we record some consequences of \cref{sec:hod_of_s}
for the structure of Nairian models.
\begin{defn}
Let \(I_\xi = \textstyle\bigcup_{\beta < \xi}\beta^\omega\).
Recall the \textit{Nairian models} from the introduction:
\[N_\xi^+ = \HOD_{I_\xi},\qquad N_\xi = N_\xi^+\cap V_\xi\]
and our convention that \(N^+\) denotes \(N_\xi^+\) and \(N\) denotes \(N_\xi\).\end{defn}

We only consider such models in the context of determinacy.
We are interested in the case that \(\xi < \Theta\) and
\(N\) is a model of \(\ZF\).
Blue--Larson--Sargsyan \cite{NairianModels} prove that this is possible;
in fact, \(\xi\) can be strongly inaccessible in \(N^+\).
Here, a cardinal \(\kappa\) is \textit{strongly inaccessible} if for all
\(\alpha<\kappa\), there is no cofinal map from \(V_\alpha\) to \(\kappa\).

\subsection{UA preliminaries}
Let \(\meas_\kappa(X)\) denote the set of \(\kappa\)-complete
    measures on \(X\) and \(\meas(X) = \meas_{\omega_1}(X)\).
\begin{defn}
    The \textit{Ketonen order} is defined on \(U,W\in\meas(\delta)\)
    by setting \(U\ke W\) if for some \(S\in W\),
    there are measures \((U_\alpha)_{\alpha \in S}\in\prod_{\alpha\in S}\meas(\alpha)\)
    such that for all \(A\subseteq \delta\), 
    \[A\in U\iff \{\alpha\in S : A\cap \alpha\in U_\alpha\}\in W\] 
\end{defn}
ZF + DC proves that the Ketonen order is a well-founded partial order \cite[Theorem 3.7]{UAFromDet}
and \(\UA\) is equivalent to the linearity of the Ketonen order \cite[Theorem 3.8]{UAFromDet}.

The linearity of the Ketonen order allows us to
define minimum measures: 
\begin{defn}[UA] For each cardinal \(\delta\), \(\mathscr K_\delta\)
denotes the Ketonen least uniform measure on \(\delta\), if one exists.
\end{defn}

\begin{lma}[{\cite[Theorem 7.2.14]{UABook}}]\label{lma:min_weakly_normal}
    If \(\delta\) is regular and \(\mathscr K_\delta\) exists, then \(\mathscr K_{\delta}\) is the unique weakly normal measure on \(\delta\) concentrating on ordinals \(\alpha\) such that \(\mathscr K_{\cf(\alpha)}\)
    does not exist.\qed
\end{lma}
Recall that a measure \(U\) on \(\delta\) is \textit{weakly normal} if every
regressive function on \(\delta\) is bounded on a \(U\)-large set.
We use this notion even in the case that \(\delta\) is a singular ordinal,
noting, however, that for singular cardinals, a different notion of weak normality is also in use
\cite[Definition 4.4.14]{UABook}.

The argument for \cref{lma:min_weakly_normal} yields:
\begin{lma}Assume \(\AD + \DC_\mathbb R\).\label{lma:ad_weakly_normal}
    Every ordinal \(\delta < \Theta\) of uncountable cofinality 
    carries a weakly normal measure that concentrates 
    on \(S^\delta_\omega\).
    \begin{proof}
        Let \(W\) be a Ketonen minimal measure
        extending the tail filter on \(\delta\).
        If \(f : \delta\to \delta\) is regressive on a \(W\)-large set,
        then \(f_*(W) \ke W\), and so \(f_*(W)\) is not uniform, which means that
        \(f_*(W)\) concentrates on a bounded set; in other words, \(f\) is bounded on a \(W\)-large set.
        This proves \(W\) is weakly normal.

        If \(S^\delta_\omega\notin W\), let \(S = \delta\setminus S^\delta_\omega\). 
        For each \(\alpha\in S\), choose
        \(U_\alpha\in \meas(\alpha)\) extending the tail filter on \(\alpha\).
        (Recall Kunen's theorem that the measures are well-ordered.)
        Put \(A\in U\) if and only if \(\{\alpha\in S : A\cap \alpha\in U_\alpha\}\in W\).
        Then \(U\ke W\), but clearly \(U\) extends the tail filter, contrary to the minimality
        of \(W\).
    \end{proof}
\end{lma}

\begin{defn}For any family of sets \(F\), \(\kappa_F\)
denotes the least cardinal \(\gamma\) such that for some sequence \(\langle A_\xi : \xi < \gamma\rangle\subseteq F\),
\(\bigcap_{\xi < \gamma}A_\xi = \emptyset\).
\end{defn}
If \(F\) is an ultrafilter, then \(\kappa_F\) is the usual completeness of \(F\); that is, \(\kappa_F = \crit(j_F)\). We will employ this more general notion in the case that \(F\) is an ultrafilter in some transitive model of set theory. 

A measure on a regular cardinal is \textit{seminormal} if it extends the closed unbounded filter.
The following theorem is one of the main results in the analysis of supercompactness
and strong compactness under \(\UA\).
\begin{thm}[{UA, \cite[Proposition 8.4.20(5)–(6), Theorem 7.3.34]{UABook}}]\label{thm:supercompact}
    Suppose \(\delta\) is a regular cardinal that carries distinct seminormal measures, and let \(\kappa = \kappa_{\mathscr K_\delta}\).
    Then \(\mathscr K_\delta\) witnesses that \(\kappa\) is \(\delta\)-strongly compact and \({<}\delta\)-supercompact. Moreover, \(\kappa\) is \(\delta\)-supercompact. \qed
\end{thm}

We will analyze measures in Nairian models by connecting them to 
canonical measures in \(\HOD_s\).
\begin{defn}
    If \(s\in \Ord^\omega\) and \(\eta\) is a cardinal in \(\HOD_s\),
    then \(\mathscr K_\eta^s\) denotes \((\mathscr K_\eta)^{\HOD_s}\), if \((\mathscr K_\eta)^{\HOD_s}\) exists.
\end{defn}

The following lemma traces back to Ketonen \cite{Ketonen1972}.
\begin{lma}\label{lma:ketonen}
    Suppose \(s\in \Ord^\omega\), \(\kappa\) is a cardinal in \(\HOD_s\), and \(\eta\in(\kappa,\Theta)\) is regular in \(\HOD_s\). If \(\HOD_s\) has the \((\aleph_1,\kappa)\)-cover property at \(\eta\), then \((\kappa_{\mathscr K_\eta})^{\HOD_s} \leq \kappa\).
    \begin{proof}
        The \((\aleph_1,\kappa)\)-cover property at \(\eta\) implies \(\eta\) has uncountable cofinality, and hence \(\mathscr K^s_\eta\) exists. 
        Note that \(S^\eta_\omega \in \mathscr K_\eta^s\) by \cref{lma:min_weakly_normal}: by the ordinal definability of measures, if \(\alpha < \Theta\) and \(\cf^{\HOD_s}(\alpha)\) carries no uniform measure in \(\HOD_s\), then \(\cf(\alpha) = \omega\).
        The \((\aleph_1,\kappa)\)-cover property then implies
        \((S^\eta_{<\kappa})^{\HOD_s} \in \mathscr K_\eta^s\).
        
        Now let \(j : \HOD_s\to M\) be the ultrapower embedding associated to \(\mathscr K_\eta^s\).
        We must show \(\crit(j)\leq \kappa\). By \L o\'s's theorem, \(\id_{\mathscr K^s_\eta} \in \left(S^{j(\eta)}_{<j(\kappa)}\right)^{M}\), or in other words:
        \[\cf^M(\sup j[\eta]) < j(\kappa)\]
        Since \(\cf^{\HOD_s}(\sup j[\eta]) = \eta\),
        we must have \(j(\kappa) \geq \eta\), and hence \(\crit(j)\leq \kappa\).
    \end{proof}
\end{lma}

\subsection{Measurable cardinals in Nairian models}\label{sec:general-measurable}
    In this subsection, we assume:
\begin{itemize}
    \item \(\AD^+ + V = L(P(\mathbb R))\).
    \item \(\xi \leq \Theta\) is strongly inaccessible in \(N^+ = N_\xi^+\).
\end{itemize}
Let \[N = N_\xi\]
Our next few observations, until \cref{thm:monotonicity}, concern a fixed ordinal \(\delta\geq \Theta^N\) that is regular in \(N\).

The first of these lemmas allows us to apply \cref{thm:supercompact} to \(\delta\):
\begin{lma}\label{lma:seminormal}
    For all sequences \(s\in I_\xi\), \(\delta\) carries distinct seminormal measures
    in \(\HOD_s\).
    \begin{proof}
        Since \(\xi\) is a strong limit cardinal in \(N^+\), 
        \((2^\delta)^{\HOD_s} < \xi\). 
        In particular, every \(\omega\)-sequence of \(\OD_s\)-subsets of \(\delta\)
        belongs to \(N\).
        Let 
        \(S_0 = S^\delta_\omega\) and \(S_1 = \delta \setminus S_0\);
        since \(\delta\geq \omega_2\), both are stationary in \(N\).
        Let \(F\) be the club filter of \(\HOD_s\) and let \(F_0 = F|S_0\) and \(F_1 = F|S_1\). 
        Then \(F_0\) and \(F_1\) extend to countably complete filters
        in \(V\), and thus extend to measures \(U_0\) and \(U_1\) in \(V\).
        By Kunen's theorem, these measures are \(\OD\),
        and therefore \(U_0\cap \HOD_s\) and \(U_1\cap \HOD_s\) witness
        that in \(\HOD_s\), there are distinct seminormal measures
        on \(\delta\).
    \end{proof}
\end{lma}

\begin{prp}\label{prp:gamma_kappa_cov}
    Suppose \(s\in I_\xi\), 
    \(\bar \kappa = (\kappa_{\mathscr K_\delta^s})^{N}\), and
    \(\kappa = (\kappa_{\mathscr K_\delta^s})^{\HOD_s}\).
    \begin{enumerate}[(1)]
        \item \(\HOD_s\) has the \((\bar \kappa,\kappa)\)-cover property at
         \(\delta\) in \(N\).
        \item \(\mathscr K_\delta^s\) witnesses that \(\kappa\) is \(\delta\)-strongly compact
        and \({<}\delta\)-supercompact in \(\HOD_s\).
        \item \(\kappa\) is \(\delta\)-supercompact in \(\HOD_s\).
    \end{enumerate}
    \begin{proof}
        Since \(\delta\) carries distinct seminormal measures, 
        we can apply \cref{thm:supercompact} 
        to conclude that \(\mathscr K_\delta^s\) witnesses that \(\kappa\) is \(\delta\)-strongly compact
        and \({<}\delta\)-supercompact in \(\HOD_s\),
        and moreover \(\kappa\) is \(\delta\)-supercompact in \(\HOD_s\).
        Since \(\mathscr K_\delta^s\) witnesses that \(\kappa\) is \(\delta\)-strongly compact in \(\HOD_s\), there is a function \(p : \delta\to P_\kappa(\delta)\) in \(\HOD_s\)
        such that \(\mathcal U = p_*(\mathscr K_\delta^s)\) is a fine measure on \(P_\kappa(\delta)\).

        The \((\bar \kappa,\kappa)\)-cover property now follows.
        Fix \(\sigma \in P_{\bar \kappa}(\delta)\cap N\).
        For \(\alpha < \delta\), define
        \[A_\alpha = \left\{\tau\in \left(P_\kappa(\delta)\right)^{\HOD_s} : \alpha\in \tau\right\}\]
        so \(A_\alpha\in \mathcal U\) by fineness.
        Since \(\mathscr K_\delta^s\) is \(\bar \kappa\)-complete,
        \(\bigcap_{\alpha\in \sigma}A_\alpha\neq \emptyset\),
        and if \(\tau\) belongs to this intersection, then \(\tau\in \HOD_s\), 
        \(\sigma\subseteq \tau\)
        and \(|\tau|^{\HOD_s} < \kappa\).
    \end{proof}
\end{prp}

\begin{prp}\label{prp:ks_complete}
    For any \(s\in I_\xi\), 
    \(\mathscr K_\delta^s\) is \(\Theta^{N}\)-complete in \(\HOD_s\).
    \begin{proof}
        Otherwise, by \cref{prp:gamma_kappa_cov},
        for some \(\kappa < \Theta^{N}\),
        \(\HOD_s\) has the \((\aleph_1,\kappa)\)-cover property at \(\delta\).
        To get a contradiction, fix \(\gamma \in (\kappa,\Theta^{N})\)
        such that \(\gamma\to (\gamma)^{<\omega_1}\). 
        By a standard argument, there is an \(\omega\)-club of \(\alpha < \gamma\) such that
        \(\alpha\) is regular in \(\HOD_s\).\footnote{Let \(U\) be
        the \(\omega\)-club filter on \(\gamma\). The partition property
        implies that \(U\) is a normal measure,
        and hence \(U\cap \HOD_s\) is a normal measure in \(\HOD_s\).
        Since \(\HOD_s\) satisfies \(\ZFC\), \(\Reg^{\HOD_s}\cap \gamma\in U\).}
        This contradicts that \(\HOD_s\) has the 
         \((\aleph_1,\kappa)\)-cover property at \(\gamma\).
    \end{proof}
\end{prp}

\begin{prp}\label{prp:W_to_K}
    Suppose that in \(N\), \(W\) is a weakly normal measure
    on \(\delta\) that concentrates on \(S^\delta_\omega\).
    Then for all \(s\in I_\xi\),
    \(W\cap \HOD_s = \mathscr K^s_\delta\).
    \begin{proof}
        This follows immediately from \cref{lma:min_weakly_normal}
        since \(W\cap \HOD_s\) is weakly normal in \(\HOD_s\)
        and by the \(\HOD_s\)-ultrafilter
        hypothesis (\cref{thm:hod_uf}), \(W\cap \HOD_s\) 
        concentrates on ordinals \(\alpha\) such that \(\mathscr K_{\bar \alpha}^s\)
        does not exist where \(\bar \alpha = \cf^{\HOD_s}(\alpha)\).
    \end{proof}
\end{prp}

\begin{cor}\label{prp:wd_generation}
    Suppose that in \(N\), \(W\) is a weakly normal measure
    on \(\delta\) that concentrates on \(S^\delta_\omega\).
    Then \(W = \bigcup_{s\in I_\xi}\mathscr K^s_\delta\)
    and \(W\) is \(\Theta^{N}\)-complete.
    \begin{proof}
        This follows from \cref{prp:ks_complete} and \cref{prp:W_to_K} 
        since every subset of \(\delta\) in \(N\) belongs
        to \(\HOD_s\) for some \(s\in I_\xi\).
    \end{proof}
\end{cor}

\begin{cor}\label{thm:unique_complete}
    In \(N\), there is a unique weakly 
    normal measure concentrating on \(S^\delta_\omega\).
    \begin{proof}
        Existence follows from \cref{lma:ad_weakly_normal}:
        since \(\delta\) is regular in \(N\),
        \(\cf(\delta) > \omega\); therefore \cref{lma:ad_weakly_normal} provides a weakly normal measure \(W\) on \(\delta\) concentrating on \(S^\delta_\omega\); since \(W\) is \(\OD\), \(W\cap N\in N\).
        Uniqueness is immediate from \cref{prp:wd_generation}.
    \end{proof}
\end{cor}

\begin{defn}\label{defn:w_delta}
    Let \(W_\delta\) denote the unique weakly normal measure of \(N\)  with \(S^\delta_\omega\in W_\delta\),
    and let \(\kappa_\delta\) denote the completeness of \(W_\delta\) as computed in \(N\).
\end{defn}
The proof of \cref{thm:unique_complete} allows us to compute the completeness of \(W_\delta\):
\begin{cor}\label{cor:ua_completeness}
    \(\kappa_\delta = \min\left\{(\kappa_{\mathscr K_\delta^s})^{\HOD_s}:s\in I_\xi\right\}\).\qed
\end{cor}
Before continuing to our general results,
we state some immediate corollaries under anti-inner model hypotheses.
\begin{prp}\label{prp:anti_inner}
    Assume the following holds in \(N\): for all cardinals \(\lambda\), there is no inner model \(M\) of \(\ZFC + \UA\) in which some \(\kappa < \lambda\) is \(\lambda\)-supercompact and such that the \(M\)-ultrafilter hypothesis holds.
    Then in \(N\), every regular cardinal is measurable.
    \begin{proof}
        This follows from \cref{prp:gamma_kappa_cov} and \cref{cor:ua_completeness}.
    \end{proof}
\end{prp}
It is unclear that it is consistent for the anti-inner model hypothesis of \cref{prp:anti_inner}
to hold in a Nairian model \(N^+_\xi\) where \(\xi\) is strongly
inaccessible. The question that arises is the following: 
\begin{qst}
    Which of the following large cardinal hypotheses has higher consistency strength over ZFC? 
    \begin{itemize}
        \item There is a cardinal \(\kappa\) that is \(\kappa^+\)-supercompact.
        \item There is a Woodin limit of Woodin cardinals.
    \end{itemize}
\end{qst}
Clearly if \(\kappa\) is \(2^\kappa\)-supercompact, then \(\kappa\) is a Woodin limit of Woodin cardinals, so we state the following proposition, whose anti-large cardinal hypothesis \textit{is} consistent: 
\begin{prp}\label{prp:anti_inner2}
    Assume that in \(N\) there is no inner model of \(\ZFC\) containing a cardinal \(\kappa\) that is \(2^{2^\kappa}\)-supercompact.
    Then in \(N\), every regular cardinal \(\delta\) carries a measure that is \(\lambda\)-complete for all cardinals \(\lambda < \delta\).
    \begin{proof}
        The proof is the same as that of \cref{prp:anti_inner}
        using the fact that under \(\UA\), if \(\kappa\) is \(\kappa^{++}\)-supercompact, then \(2^{2^\kappa} = \kappa^{++}\).
    \end{proof}
\end{prp} 
We now continue to the theorems that require no anti-inner model hypothesis.

\begin{prp}\label{prp:kappa_cov}
    Let \(\kappa = \kappa_\delta\), and suppose \(s\in I_\xi\) is such that \((\kappa_{\mathscr K^s_\delta})^{\HOD_s} = \kappa\). Then
    \(\HOD_s\) has the \(\kappa\)-cover and approximation properties at \(\delta\) in \(N\). Moreover, 
    every \(S\in P(\delta)\cap N\) is in
    \(\HOD_{s,x}\) for some \(x\in P(\kappa)\cap N\).
    \begin{proof}
        The \(\kappa\)-cover property
        follows immediately from \cref{prp:gamma_kappa_cov}.
        For the approximation property, 
        fix \(\mathcal A = \langle A_\alpha : \alpha < \delta\rangle\in \HOD_s\) such that
        \(\HOD_s\) satisfies that \(\mathcal A\) is a \(\kappa\)-independent family 
        of subsets of \(\delta\). 
        Suppose \(T\in P(\delta)\cap N\) is \(\kappa\)-approximated by \(\HOD_s\).
        Then the basis
        \[\mathcal B = \{A_\alpha : \alpha \in T\} \cup \{\delta\setminus A_\alpha : \alpha\notin T\}\]
        extends to a 
        \(\kappa\)-complete filter \(F\) in \(N\). 
        In particular, \(F\) 
        is countably complete in \(V\), and thus \(F\) extends to a measure \(U\) in \(V\).
        By Kunen's theorem, \(U\) is \(\OD\), and so \(U\cap \HOD_s\in \HOD_s\).
        Thus
        \(T = \{\alpha < \delta : A_\alpha\in U\}\in \HOD_s\).

        To show every \(S\in P(\delta)\cap N\) is in
        \(\HOD_{s,x}\) for some \(x\in P(\kappa)\cap N\),
        fix \(t\in I_\xi\) such that \(s\) and \(S\) belong to \(\HOD_t\).
        Since \(\mathscr K^s_\delta\subseteq \mathscr K^t_\delta\),
        \((\kappa_{\mathscr K^t_\delta})^{\HOD_t} = \kappa\).
        In particular, \(\kappa\) is \(\delta\)-strongly compact in 
        \(\HOD_t\). Let \(x\subseteq \kappa\) code a well-order of \(P_\kappa(\kappa)\cap \HOD_t\). 
        Note that the models \(\HOD_{s,x}\) and \(\HOD_t\) compute \(P_\kappa(\delta)\) the same way:
        of course, \(\HOD_{s,x}\subseteq \HOD_t\), and in the other direction,
        if \(\sigma\in (P_\kappa(\delta))^{\HOD_t}\), then \(\sigma\) is covered by \(\tau\in (P_\kappa(\delta))^{\HOD_s}\) by the \(\kappa\)-cover property, \(P(\tau)\cap \HOD_t\subseteq \HOD_{s,x}\) since \(P_\kappa(\kappa)\cap \HOD_t\subseteq \HOD_{s,x}\). The same argument shows that
        \[(\delta^{<\kappa})^{\HOD_{s,x}} \leq (\delta^{<\kappa})^{\HOD_s}\cdot (\kappa^{<\kappa})^{\HOD_{s,x}} = \delta\cdot \kappa = \delta\]
        
        This cardinal arithmetic yields \(\mathcal A = \langle A_\alpha : \alpha < \delta\rangle\in \HOD_{s,x}\) such that
        \(\HOD_{s,x}\) satisfies that \(\mathcal A\) is a \(\kappa\)-independent family 
        of subsets of \(\delta\). Since \(\HOD_{s,x}\) and \(\HOD_t\) 
        have the same \(P_\kappa(\delta)\),
        \(\mathcal A\) is also \(\kappa\)-independent in \(\HOD_t\).
        By \cref{prp:gamma_kappa_cov}, 
        \(\kappa\) is \(\delta\)-strongly compact in \(\HOD_t\).
        Repeating the argument of the previous paragraph, we can find a \(\kappa\)-complete measure
        \(U\) of \(\HOD_t\) such that 
        \(S = \{\alpha < \delta : A_\alpha\in U\}\), and using the \(\HOD_t\)-ultrafilter hypothesis (\cref{thm:hod_uf}), we obtain that \(U\cap \HOD_{s,x}\in \HOD_{s,x}\), and therefore \(S\in \HOD_{s,x}\).
    \end{proof}
\end{prp}

\begin{thm}
\label{thm:monotonicity}
   If \(\delta_0 \leq \delta_1 < \xi\) are regular cardinals in \(N\), then \(\kappa_{\delta_0}\leq \kappa_{\delta_1}\).
    \begin{proof}
        By the Steel--Woodin theorem \cite[Theorem 2.18]{SteelOrdinalDefinability}, we may assume \(\delta_0 \geq \Theta^{N}\).
        By \cref{cor:ua_completeness},
        it suffices to show that for all \(s\in I_\xi\),
        \(\mathscr K_{\delta_1}^s\) is \(\kappa\)-complete in \(\HOD_s\) where
        \(\kappa = \kappa_{\delta_0}\). 

        Fix \(s\in I_\xi\). A consequence of \cite[Lemma 7.4.49]{UABook} is that under \(\ZFC + \UA\), if \(\gamma_0 \leq \gamma_1\) are regular non-isolated Fr\'echet cardinals, then \(\kappa_{\mathscr K_{\gamma_0}}\leq \kappa_{\mathscr K_{\gamma_1}}\).
        By \cref{lma:seminormal},
        both \(\delta_0\) and \(\delta_1\) carry distinct 
        seminormal measures in \(\HOD_s\), and so
        by \cite[Theorem 8.4.20]{UABook}, 
        they are regular non-isolated Fr\'echet cardinals in \(\HOD_s\).
        By \cref{prp:W_to_K}, it follows that
        \(\kappa \leq \kappa_{\mathscr K^s_{\delta_0}}\leq \kappa_{\mathscr K^s_{\delta_1}}\).
    \end{proof}
\end{thm}

We now prove our main theorem on measurable cardinals in Nairian models.

\begin{thm}\label{thm:main}
    Let \(\mathcal H = V_\xi^\HOD\). Inside \(N\), suppose \(\eta\) is a regular cardinal satisfying: 
    \begin{enumerate}[(1)]
        \item \label{item:s_reflection} 
        For all \(\alpha < \eta\), \(P(\alpha) \subseteq \mathcal H({\eta^\omega})\). 
        \item \label{item:small_gen}  For all \(\alpha < \eta\), \(\Theta(\alpha^\omega)\leq \eta\).
    \end{enumerate}
    Then in \(N\), either \(\kappa_\eta = \eta\) or \(\mathcal H\) correctly computes all finite successors of \(\eta^{++}\).
    \begin{proof}
        By \ref{item:small_gen}, \(\Theta^{N} \leq \eta < \xi\),
        which puts us in a position to apply the previous results of this section 
        with \(\delta = \eta\).

        Assume towards a contradiction that the conclusion fails. 
        In other words, \(\kappa < \eta\) where \(\kappa = \kappa_\eta\).
        Fix \(s\in I_\xi\) such that \((\kappa_{\mathscr K_\eta^s})^{\HOD_s} = \kappa\). 

        By \cref{prp:kappa_cov},
        \(\HOD_s\) has the 
        \(\kappa\)-cover property at \(\delta\) in \(N\) and every \(S\in P(\eta)\cap N\) is in
        \(\HOD_{s,x}\) for some \(x\in P(\kappa)\cap N\).
        Combining \ref{item:s_reflection} and \ref{item:small_gen}, for any \(x\in P(\kappa)\cap N\),
        \(\HOD_{s,x}\) is an \(\eta\)-cc generic extension of \(\HOD_s\). Since every
        subset of \(\eta\) in \(N\) belongs to some such \(\HOD_{s,x}\), it follows that
        \((\eta^+)^{\HOD_s} = (\eta^+)^{N}\).

        Since \(\HOD_s\) has
        the \(\kappa\)-cover property at \(\eta\) and \((\eta^+)^{\HOD_s} = (\eta^+)^{N}\), 
        \(\HOD_s\) has the \(\kappa\)-cover property at \((\eta^+)^{N}\).
        Let \(\gamma = (\eta^+)^{N}\).
        Since \(\gamma\) is a cardinal in \(N\) that is 
        seen to be a successor cardinal by an inner model of choice,
        \(\gamma\) is regular in \(N\). 
        Applying \cref{lma:ketonen} and \cref{thm:monotonicity},
        \[\kappa = (\kappa_{\mathscr K^s_\eta})^{\HOD_s} \leq (\kappa_{\mathscr K^s_\gamma})^{\HOD_s} \leq \kappa\] 
        Thus we can apply \cref{prp:kappa_cov}
        again, to obtain that every \(S\in P(\gamma)\cap N\) is in
        \(\HOD_{s,x}\) for some \(x\in P(\kappa)\cap N\), and thus
        \((\gamma^+)^{\HOD_s} = (\gamma^+)^{N}\).

        Iterating this argument, we have that \(\HOD_s\) correctly computes \((\eta^{+n})^{N}\)
        for all \(n < \omega\).
        Let \(\lambda = (\eta^{+\omega})^{N}\). 
        By \cref{prp:gamma_kappa_cov}, \(\kappa\) is \({<}\lambda\)-supercompact in \(\HOD_s\) and
        \(\HOD_s\) has the \(\kappa\)-cover and approximation properties
        below \(\lambda\). Indeed, this holds in \(\HOD_{s,u}\) for all \(u\in I_\xi\).

        We claim that \(\Theta^{N}(P(\eta)) \leq (\gamma^+)^{N}\).
        Note first that since UA implies GCH in regions of supercompactness, 
        \((2^\eta)^{\HOD_{s,u}} = \gamma\) for all \(u\in I_\xi\).
        For each \(u\in \eta^\omega\), 
        let \(\langle S_\alpha^u : \alpha < \gamma\rangle\) be an \(\OD_{s,u}\)-enumeration of 
        \(P(\eta)\cap \HOD_{s,u}\). 
        Note that each \(S\subseteq \eta\) is of the form \(S_\alpha^u\) for some \(u\in \eta^\omega\)
        and \(\alpha < \gamma\),
        since \(S\in \HOD_{s,x}\) for some \(x\in P(\kappa)\cap N\) and
        \(x\in \HOD_u\) for some \(u\in \eta^\omega\) by \ref{item:s_reflection}.

        Suppose \(\nu < \Theta^{N}(P(\eta))\),
        and let us show \(\nu < (\gamma^+)^{N}\). 
        Suppose \(f : P(\eta)\to \nu\) is a surjection.
        Define \(g : \eta^\omega\times \gamma\to \nu\) by
        setting \(g(u,\alpha) = f(S^u_\alpha)\). Then \(g\) is also a surjection.
        Note however that for each fixed \(\alpha < \gamma\),
        if \(A_\alpha = \{g(u,\alpha) : u\in \eta^\omega\}\), then
        \(|A_\alpha|\leq \eta\) by \ref{item:small_gen}.
        (Here we use that \(\Theta^{N}(\eta^\omega) = ({\eta^+})^{N}\) as a consequence
        of \ref{item:small_gen}.)
        Since \(\nu = \bigcup_{\alpha < \gamma} A_\alpha\),
        \(\nu < (\gamma^+)^{N}\). 

        Fix \(y\subseteq \eta\) such that \(\HOD_s\cap P_\kappa(\eta)\) is
        coded by \(y\). 
        Since \(\Theta^{N}(P(\eta)) \leq (\gamma^+)^{N}\),
        \(\HOD_y\) is contained in a \((\gamma^+)^{N}\)-cc extension of
        \(\HOD\). Moreover, applying the Hamkins uniqueness theorem
        \cite[Theorem 4.7]{Goldberg2024StronglyCompactOD} at each finite successor of
        \(\eta\), we see that \(Y = \Pbd(\lambda)\cap \HOD_s\) is \(\OD_y\).
        In particular, \(\HOD_y^{L(Y)}\subseteq \HOD_y\). 
        Note that \(L(Y)\) is a model of \(\ZFC + \UA\) in which 
        \(\kappa\) is \({<}\lambda\)-supercompact; the fact that
        \(L(Y)\vDash \AC\) follows from \cite[Theorem 6.2.19]{UABook} applied in \(\HOD_s\), which implies that there is a well-order of \(Y\) that is definable from parameters over \(Y\).
        By \cite[Theorem 6.2.19]{UABook}, it follows that 
        \(Y\subseteq \HOD_y^{L(Y)}\subseteq \HOD_y\).
        This implies that \(Y\) is contained in a \((\gamma^+)^{N}\)-cc extension
        of \(\HOD\). It follows that \(\HOD\) correctly computes all 
        finite successors of \((\gamma^+)^{N} = (\eta^{++})^{N}\).
    \end{proof}
\end{thm}

Note that if \(\HOD\) correctly computes all finite successors of \((\eta^{++})^{N}\), then 
there is an ordinal \(\kappa \leq (\eta^{++})^{N}\) that 
is \((\eta^{+n})^{N}\)-supercompact in \(\HOD\) 
for all \(n < \omega\). 
In particular, this is impossible under HPC.

\subsection{Measurable cardinals in minimal Nairian models}
\label{sec:minimal-hierarchy}
Throughout this subsection, assume
\[
    \ZF+\AD_{\mathbb R}+\Theta_\textnormal{reg}
    +\textnormal{HPC}+V=L(P(\mathbb R)).
\]
Fix \(\xi < \Theta\) such that \(\xi\) is strongly
inaccessible in \(N^+ = N_\xi^+\).
We also assume that
\(N = N_\xi\)
is a minimal Nairian model in the sense of
\cite[Definition 1.1]{FailureSquare}.
We let \(\mathcal H\) denote \(V_\xi^{\HOD}\) as computed in
the ambient determinacy model.

\begin{lma}\label{lma:power_correct}
    In \(N\), if \(\alpha\geq \omega\)
    and \(\eta = \Theta(V_{\alpha+1})\),
    then for all \(\beta < \eta\), \(P(\beta) \subseteq \mathcal H(\eta^\omega)\).
    \begin{proof}
        If \(\alpha = \omega\), the lemma follows from
        \(\AD_\mathbb R + \DC\). 
        Assume \(\alpha > \omega\).
        Fix \(\beta\) such that \(\Theta \leq \beta < \eta\), and let \(\bar \eta\) be the least strong of \(\mathcal H\) above \(\beta\). By \cite[Theorem 19.1]{FailureSquare},
        \(\bar \eta = \Theta(\beta^\omega)\leq \eta\). By \cite[Theorem 10.13]{NairianModels},
        \(P(\beta)\subseteq L(\mathcal H|\bar \eta,I_{\bar \eta})\subseteq \mathcal H(\eta^\omega)\).
    \end{proof}
\end{lma}

\begin{lma}\label{lma:V_vs_I}
    In \(N\), for all \(\alpha\geq \omega\),
    \(|V_{\alpha}| =^* |I_\lambda|\)
    where \(\lambda = \sup_{\beta < \alpha} \Theta(V_\beta)\).
    \begin{proof}
        Work in \(N\). The nontrivial direction is
        \(|V_\alpha| \leq^* |I_\lambda|\).
        We produce the required surjections uniformly in \(\alpha\) by induction. 
        This uniformity handles the limit steps.
        For the successor step, 
        suppose \(|V_{\alpha}| \leq^* |I_\lambda|\), and let us show that \(|V_{\alpha+1}| \leq^* |\Theta(V_\alpha)^\omega|\).

        First, assume \(\cf(\alpha) = \omega\).
        Then \(|V_{\alpha+1}| =^* |V_\alpha^\omega| \leq^* |(I_\lambda)^\omega| =^* |\lambda^\omega|\leq^* |\Theta(V_\alpha)^\omega|\). 
        
        Otherwise, let \(\theta = \Theta(\lambda^\omega)\), and note that \(\theta \leq \Theta(V_\alpha)\) since \(|(V_\alpha)^\omega| = |V_\alpha|\).
        We use that 
        \(P(\lambda^\omega)\subseteq L_\theta(\mathcal H|\theta,I_\theta)\) by
        \cite[Theorems 10.6 and 10.13]{NairianModels}.
        Since every subset of \(V_\alpha\) can be coded by a subset of \(I_\lambda\),
        it follows that \(V_{\alpha+1}\subseteq L_\theta(\mathcal H|\theta,I_\theta)\).
        Thus \(|V_{\alpha+1}| \leq^* |I_\theta|\leq^* |I_{\Theta(V_\alpha)}|\), as desired.
    \end{proof}
\end{lma}

We can now connect the cardinals
\(\Theta(V_\alpha)\) to the strong cardinals of \(\mathcal H\):
\begin{prp}\label{prp:strongs}
    The set of strong cardinals of \(\mathcal H\) is equal to
    \[S = \{\Theta(V_\alpha) : \alpha \geq \omega\text{ and }\alpha\in \xi\setminus S^\xi_\omega\}\]
    \begin{proof}
        Fix \(\alpha\geq \omega\) and let
        \(\lambda = \sup_{\beta < \alpha} \Theta(V_\beta)\).
        We will show that the least strong of \(\mathcal H\) above \(\lambda\) is \(\min(S\setminus (\lambda+1))\).
        By \cite[Theorem 19.1]{FailureSquare}, the least strong of \(\mathcal H\) above \(\lambda\) is \(\Theta(\lambda^\omega)\).
        If \(\cf(\alpha) = \omega\), \cref{lma:V_vs_I} implies
        \(\Theta(V_{\alpha+1}) = \Theta(\lambda^\omega)\),
        and clearly \(\Theta(V_{\alpha+1}) = \min(S\setminus (\lambda+1))\). Otherwise,
        \(\Theta(V_\alpha) = \Theta(\lambda^\omega)\),
        and clearly \(\Theta(V_{\alpha}) = \min(S\setminus (\lambda+1))\).
    \end{proof}
\end{prp}

\begin{thm}\label{thm:theta_meas}
    In \(N\), if \(\alpha\geq \omega\) and \(\kappa = \Theta(V_\alpha)\), then 
    \(W_{\kappa}\) is \(\kappa\)-complete, so
    \(\kappa\) is measurable.
    \begin{proof}
        We first consider the case that \(\alpha\) is a successor ordinal. Let \(\lambda = \Theta(V_{\alpha-1})\).
        Since \(\lambda\) has uncountable cofinality, \cref{lma:V_vs_I} implies \(\Theta(V_{\alpha}) = \Theta(\lambda^\omega)\).
        Since \(\Theta(V_{\alpha})\) is a strong cardinal of \(\mathcal H\), it
        is regular by \cite[Theorem 10.20]{NairianModels}.

        Apply \cref{thm:main} in \(N\)
        to \(\eta = \kappa\).
        We check the three criteria from that theorem.
        It is easy to see that for all \(\beta < \kappa\),
        \(\Theta(\beta^\omega)\leq \kappa\).
        By \cref{lma:power_correct},
        for all \(\beta < \kappa\), \(P(\beta) \subseteq \mathcal H(\kappa^\omega)\).
        Finally, \(\HOD\) does not correctly compute any successor cardinals by our HPC assumption \cite{Steel2022Comparison}.

        For the limit case, we use the following claim:
\begin{clm}\label{lma:limit-successor-regular}
    In \(N\), for every limit ordinal \(\alpha\), \[\Theta(V_\alpha) = \left(\sup_{\beta < \alpha} \Theta(V_\beta)\right)^+\] and \(\Theta(V_\alpha)\) 
    is regular.
\end{clm}
\begin{proof}
    It is a consequence of ZF that \(\Theta(V_\alpha) = (\sup_{\beta < \alpha} \Theta(V_\beta))^+\);
    see \cite[Lemma 2.1]{GoldbergCCCP}.
    Work in \(N\). We may assume \(\alpha > \omega\).
    Let \(\lambda = \sup_{\beta < \alpha} \Theta(V_\beta)\).
    Since \(N\) is a minimal Nairian model, 
    it contains no strongly inaccessible cardinals, 
    and so \(\lambda\) is singular.

    Assume towards a contradiction that \(\lambda^+\) is singular, say of cofinality \(\mu < \lambda\).
    Let \(\nu\) be a sufficiently large ordinal.
    By Blue--Larson--Sargsyan
    \cite[Theorem 1.17 and Definitions 6.2 and 8.6]{FailureSquare},
    there is an elementary substructure
    \(X\preceq V_\nu\) such that \(\mu\subseteq X\),
    \(\lambda \in X\),
    and \(|X| \leq^* |V_\beta|\) for some \(\beta < \alpha\). (We use here that one can bound the sizes of levels of the \(\mathsf{N}\)-hierarchy from \cite{FailureSquare} by levels of the  cumulative hierarchy; for example, one can prove by induction that \(|\mathsf{N_\gamma}|\leq^* |V_{\omega+\gamma+2}|\).)

    We now use the argument of \cite[Lemma 3]{UsubaChoicelessLS}.
    Since \(\mu\subseteq X\), \(X\cap\lambda^+\) is cofinal in \(\lambda^+\).
    Consider the partial map \(e:X\times\lambda\to\lambda^+\)
    given by \(e(h,\alpha) = h(\alpha)\) if \(h : \lambda \to \lambda^+\)
    is a function in \(X\).
    For each \(\nu<\lambda\), let
    \(R_\nu=\{e(h,\nu):h\in X\}\),
    so \(\lambda^+ = \bigcup_{\nu < \lambda}R_\nu\) and \(\ot(R_\nu) < \lambda\)
    since \(\Theta(X) \leq \lambda\).
    Let \(\varphi_\nu : \ot(R_\nu)\to R_\nu\) be the increasing enumeration,
    and define a partial function \(\Phi: \lambda\times\lambda\to \lambda^+\)
    by \(\Phi(\nu,\alpha) = \varphi_\nu(\alpha)\).
    Then \(\Phi\) surjects \(\lambda\times \lambda\) onto \(\lambda^+\), a contradiction.
\end{proof}

    We now consider the case that \(\alpha\) is a limit ordinal.
        Let \(\lambda = \sup_{\beta < \alpha} \Theta(V_\beta)\). 
        By \cref{lma:limit-successor-regular},
        \(\Theta(V_\alpha) = \lambda^+\) is regular.
        By \cref{thm:monotonicity} and the successor case,
        \(\lambda \leq \kappa_{\lambda^+} \leq \lambda^+\).
        The minimality of \(N\) implies that \(\lambda\) is singular.
        Since \(\kappa_{\lambda^+}\) is regular, \(\kappa_{\lambda^+} = \lambda^+\).
    \end{proof}
\end{thm}

\begin{cor}\label{cor:consistency}
    The theory ``\textnormal{ZF + }there are
    infinitely many consecutive measurable cardinals'' is consistent relative to a Woodin limit of Woodin cardinals.
    \begin{proof}
        Blue--Larson--Sargsyan 
        construct minimal Nairian models
        starting with a Woodin limit of Woodin cardinals. In these models,
        their calculations \cite[Theorem 1.3(1)--(3)]{FailureSquare},
        together with \cref{lma:V_vs_I}, imply that
        \(\Theta(V_{\omega+n+1}) = \Theta^{+n}\) for all \(n < \omega\).
        By \cref{thm:theta_meas}, this yields
        infinitely many consecutive measurable cardinals.
    \end{proof}
\end{cor}

We now show
minimal Nairian models satisfy \(\aleph(V_{\alpha+1})=\Theta(V_\alpha)\) for all infinite ordinals \(\alpha\). In other words, there is no
\(\Theta(V_\alpha)\)-sequence of distinct subsets of \(V_\alpha\).  

We will use the following abstract forcing lemma \cite[Lemma 5.2]{HamkinsWoodin2005NecessaryMaximality}, due to Jech \cite[p.~5]{Jech1974SimpleBooleanAlgebras}.\footnote
    {The suggestion to use this result to bound the size of the Vop\v{e}nka algebra is due to GPT-6 Astra.}

\begin{lma}[ZFC]\label{lma:small-completely-generated-algebra}
    Suppose \(\kappa\) is weakly compact and \(\mathbb B\) is a
    \(\kappa\)-cc complete Boolean algebra completely generated by
    fewer than \(\kappa\) elements. Then \(|\mathbb B|<\kappa\).\qed
\end{lma}

\begin{thm}\label{thm:woodin_vopenka}
    Suppose \(X\) is a transitive set and \(M\) is an inner model of \(\ZFC\) such that \(M = \HOD_{M\cup \{X\}}\).
    Let \(g \subseteq \Col(\omega,X)\) be \(V\)-generic and let \(\kappa = \Theta(X^{<\omega})\). Then
    \(g\) belongs to a \(\kappa\)-cc forcing extension of \(M\).
\end{thm}
\begin{proof}
    We verify that
    \(M\) has the uniform \(\kappa\)-cover property in 
    \(M_g = \HOD^{V[g]}_{M\cup \{g\}}\) and then cite Bukovsky's criterion
    \cite[Proposition 4.37]{GoldbergUniqueness}.
    Suppose \(I\in M\) and \(f : I \to \Ord\) is a function in \(M_g\). Fix a formula \(\varphi\) and a parameter \(a\in M\) such that \[f(i) = \alpha\iff V[g]\vDash \varphi(i,\alpha,g,a)\]
    For each \(i\in I\), let 
    \[F(i) = \{\alpha\in \Ord : \exists p\in \Col(\omega,X)\, \left(p\Vdash \varphi(i,\alpha,g,a)\right)\}\]
    Then \(F\in M\), and for all \(i \in I\), \(f(i)\in F(i)\) and \(|F(i)| \leq^* |X^{<\omega}|\), so \(\ot(F(i)) < \kappa\).
\end{proof}

\begin{prp}\label{prp:measurable-preserves-local-regularity}
    Suppose \(\kappa\) is measurable and \(X\) is a transitive
    set with \(\Theta(X)\leq \kappa\). Then, for
    any inner model \(M\) of \(\ZFC\), the following hold:
    \begin{enumerate}[(1)]
        \item \(\kappa\) is strongly inaccessible in
        \(M(X)\).
        \item For any normal measure \(W\) on \(\kappa\),
        \(W\)-almost every \(\rho<\kappa\) remains regular in
        \(M(X)\).
    \end{enumerate}
\end{prp}
\begin{proof}
    Fix a measure \(W\) witnessing that \(\kappa\) is measurable.
    Let \(M'\) be the class of sets hereditarily
    ordinal definable from parameters in \(M\cup \{X,W\}\). Then \(M' = \HOD_{M'\cup \{X\}}\). Since \(M\subseteq M'\), it suffices to prove (1) and (2) for \(M'(X)\).

    Applying \cref{thm:woodin_vopenka},
    let \(\mathbb A\) be a \(\kappa\)-cc complete Boolean  
    algebra and \(G\subseteq \mathbb A\) be an \(M'\)-generic filter such that \(M'[G]\) contains a \(V\)-generic collapse \(g : \omega\to X\).
    Let \(\dot x\) be an \(\mathbb A\)-name for
    a real in \(M'[G]\) coding
    \(X\); for example, \(\dot x\) could code the pull-back
    of the membership relation \(\in\restriction X\) under \(\dot g\) to a binary relation on \(\omega\).
    Let \(\mathbb B\) be the complete subalgebra of 
    \(\mathbb A\) generated by the Boolean values of 
    \(\dot x(n) = m\) for \(n,m < \omega\).
    Then in \(M'\), \(\mathbb B\) is a countably generated \(\kappa\)-cc Boolean algebra and
    \(M'(X)\) is contained in a \(\mathbb B\)-generic extension of \(M'\).

    Since \(\kappa\) is measurable in \(M'\), \cref{lma:small-completely-generated-algebra} implies \(|\mathbb B| < \kappa\).
    Therefore \(\kappa\) remains strongly inaccessible in any \(\mathbb B\)-generic extension of \(M'\), and so \(\kappa\) is strongly inaccessible in 
    \(M'(X)\), which proves (1).
    For (2), note that if \(W\) is normal, then \(W\) concentrates on regular cardinals of \(M'\), and forcing with \(\mathbb B\) preserves all sufficiently large \(M'\)-regular cardinals below \(\kappa\).
\end{proof}

For the next corollary, we return to the setting of
\cref{sec:minimal-hierarchy}: assume
\[
    \ZF+\AD_{\mathbb R}+\Theta_\textnormal{reg}
    +V=L(P(\mathbb R))+\textnormal{HPC}.
\]
Suppose \(\xi < \Theta\) is inaccessible in \(N^+ = N_\xi^+\)
and \(N = N_\xi\) is a minimal Nairian model. 

\begin{cor}\label{cor:no-long-rank-sequences}
    In \(N\), if \(\omega\leq \alpha < \xi\), there is no
    \(\Theta(V_\alpha)\)-sequence of distinct subsets of
    \(V_\alpha\).
\end{cor}
\begin{proof}
    Let \(\kappa=\Theta(V_\alpha)\), which is measurable by
    \cref{thm:theta_meas}.

    First suppose that \(\alpha\) is a successor ordinal.
    Assume towards
    a contradiction that 
    \(\vec A=\langle A_\nu:\nu<\kappa\rangle\)
    is a sequence of distinct subsets of \(V_\alpha\).
    By \cref{lma:V_vs_I}, for some \(\eta<\kappa\),
    there is a surjection
    \(f:\eta^\omega\to V_\alpha\). (We can take
    \(\eta = \Theta(V_{\alpha-1})\).)
    Let
    \(B_\nu=f^{-1}[A_\nu]\) for \(\nu<\kappa\).
    Choose \(s\in I_\xi\) such that \(f\) and
    \(\vec A\) are \(\OD_s\). By
    \cite[Theorem 12.11]{FailureSquare}, the relation
    \[
        \{(\nu,t)\in\kappa\times\eta^\omega:t\in B_\nu\},
    \]
    has an \(\infty\)-Borel code in \(\HOD_s\). Hence
    \(\langle B_\nu:\nu<\kappa\rangle\) belongs to
    \(\HOD_s(\eta^\omega)\).
    By \cref{prp:measurable-preserves-local-regularity}(1),
    \(\kappa\) is strongly inaccessible in \(\HOD_s(\eta^\omega)\),
    contradicting that
    \(\kappa\) injects into \(P(\eta^\omega)\)
    via the map \(f(\nu) = B_\nu\).

    Now suppose that \(\alpha\) is a limit ordinal, and let \(U\) be a nonprincipal
    \(\kappa\)-complete ultrafilter on \(\kappa\). 
    The successor case shows that if
    \(\beta<\alpha\), there is no
    \(\kappa\)-sequence of distinct subsets of \(V_\beta\).
    By \cite[Lemma 3.5]{Goldberg2024MeasurableChoiceless},
    \(U\) is therefore closed under
    \(V_\beta\)-indexed intersections. Since \(\alpha<\kappa\) and \(U\) is \(\kappa\)-complete, this implies that \(U\) is
    closed under \(V_\alpha\)-indexed intersections as well.
    A standard argument now
    shows that there is no \(\kappa\)-sequence of distinct subsets of \(V_\alpha\); see the proof of \cite[Lemma 10.4]{Jech2003SetTheory}.
\end{proof}

\subsection{Nairian models, supercompact cardinals, and Gitik's theorem}\label{sec:gitik}
In this subsection (except \cref{prp:not_vacuous}), we assume:
\begin{itemize}
    \item \(\AD^+ + V = L(P(\mathbb R))\).
    \item \(\xi \leq \Theta\) is strongly inaccessible in \(N^+ = N_\xi^+\).
\end{itemize}
This means we can apply the results
of \cref{sec:general-measurable} to any \(\delta\) that is regular in \(N = N_\xi\)
and satisfies \(\Theta^{N}\leq \delta < \xi\).
In particular, \(W_\delta\) is defined for all such \(\delta\); see \cref{defn:w_delta}.

\begin{thm}\label{thm:dichotomy}
    Exactly one of the following holds:
    \begin{enumerate}[(1)]
        \item \label{item:eventual_complete} \(N\) satisfies
        that for all \(\kappa\), for all sufficiently large regular \(\delta\),
        \(W_\delta\) is \(\kappa\)-complete.
        \item \label{item:supercompact} \(\HOD\cap V_\xi\) satisfies that there is a supercompact cardinal.
    \end{enumerate}
    \begin{proof}
        Work in \(N^+\).
        Assume \ref{item:eventual_complete} fails. By \cref{thm:monotonicity},
        there is some cardinal \(\kappa_*\) such that for all sufficiently large
        regular \(\delta < \xi\), \(\kappa_\delta = \kappa_*\).
        
        By \cref{prp:gamma_kappa_cov} and \cref{prp:kappa_cov}, for all regular \(\delta \in [\kappa_*,\xi)\), 
        there is some \(s\in I_\xi\) such that
        \(\kappa_*\) is \(\delta\)-supercompact in \(\HOD_s\)
        and \(\HOD_s\) has the \(\kappa_*\)-cover 
        and approximation properties at \(\delta\).
        Let \(\gamma = \kappa_*^+\). 
        By the Hamkins uniqueness theorem \cite[Theorem 4.7]{Goldberg2024StronglyCompactOD},
        if \(M_0\) and \(M_1\) are inner models of ZFC with the
        \(\kappa_*\)-approximation and cover properties at \(\delta\)
        such that \(M_0\cap H(\gamma) = M_1\cap H(\gamma)\),
        then \(M_0 \cap P(\delta) = M_1\cap P(\delta)\).
        By a pigeonhole argument, using that \(\xi\) is inaccessible,
        we can find a single transitive model \(M\subseteq V_\xi\)
        with the \(\kappa_*\)-cover and approximation properties
        such that for arbitrarily large regular \(\delta\),
        there is some \(s\in I_\xi\)
        such that \(P(\delta)\cap M = P(\delta)\cap \HOD_s\)
        and \(\kappa_*\) is \(\delta\)-supercompact in \(\HOD_s\).

        It follows easily that \(\kappa_*\) is supercompact in \(M\)
        and \(M\) satisfies \(\UA\). Moreover
        \(M\) is ordinal definable from \(x = M\cap H(\gamma)\).
        Thus \(\HOD^M\subseteq \HOD_x\). 
        Note also that \(\HOD^M\) has the \(\kappa_*\)-cover and approximation properties in \(M\) 
        since \(M\vDash \UA\) and \(\UA\) implies that \(\HOD\) is a weak extender
        model at each supercompact cardinal \cite[Theorem 6.2.15]{UABook}.

        We now leave \(N^+\) and work in the ambient determinacy model. 
        We have that \(\HOD^M\subseteq \HOD_x^{N^+}\subseteq \HOD_x\).
        By Vop\v{e}nka's theorem, \(\HOD_x\) is contained in a \(\rho\)-cc
        generic extension of \(\HOD\) for some \(\rho < \xi\). 
        (A coarse upper bound is \(\rho = (\Theta(V_{\gamma+1}))^N\).)
        Since \(\HOD^M\) has the \(\kappa_*\)-cover property below \(\xi\) 
        and is contained in a \(\rho\)-cc generic extension of \(\HOD\), it follows that \(\HOD\) has the
        \((\kappa_*,\rho)\)-cover property below \(\xi\). This covering property
        implies that for all regular \(\delta\in (\rho,\xi)\), \(\HOD\) satisfies
        \(\kappa_{\mathscr K_\delta}\leq \rho\). Applying \cref{prp:gamma_kappa_cov},
        it follows immediately that there is a supercompact cardinal in \(\HOD\cap V_\xi\).

        Finally, we sketch the proof that \ref{item:eventual_complete} and \ref{item:supercompact}
        are mutually exclusive. Assume \ref{item:supercompact} holds.
        We will show that \(\xi\) is a limit of regular cardinals in \(N^+\).
        Granting this, let us show that \ref{item:eventual_complete} fails. 
        Let \(\kappa_0\) be the least supercompact of \(\HOD\cap V_\xi\). 
        By \cite[Lemma 7.4.49]{UABook} if \(\kappa_0\leq \delta < \xi\) and \(\delta\) is a regular cardinal in \(\HOD\),
        then \(\mathscr K_\delta^0 = (\mathscr K_\delta)^\HOD\) has completeness
        \(\kappa_0\) in \(\HOD\).
        If \(\delta > \kappa_0\) is regular in \(N\), then by \cref{prp:W_to_K}
        \(\kappa_\delta\leq \kappa_{\mathscr K_\delta^0} = \kappa_0\), which rules out \ref{item:eventual_complete}.
        
        We now show that \(\xi\) is a limit of regular cardinals in \(N^+\).
        First, note that \cref{prp:gamma_kappa_cov} shows that \(\HOD\) has the \((\aleph_1,\kappa_0)\)-cover
        property below \(\xi\). 
        Thus the same is true for \(\HOD_s\) for any \(s\in I_\xi\),
        and it follows from \cref{lma:ketonen} that 
        there is a supercompact cardinal in \(\HOD_s\cap V_\xi\).
        Let us denote by \(\kappa_s\) the least supercompact cardinal of \(\HOD_s\cap V_\xi\). 
        Thus \(\HOD_s\) has the \((\aleph_1,\kappa_s)\)-cover property below \(\xi\).

        Fix \(s\in I_\xi\) with \(\kappa_s\) as small as possible.
        We claim that \(\kappa_s\) is regular in \(N\):
        otherwise, take \(t\in I_\xi\) such that \(s\in \HOD_t\) and
        \(\kappa_s\) is singular in \(\HOD_t\). Since \(\HOD_s\subseteq \HOD_t\), 
        \(\kappa_t\leq \kappa_s\), but since \(\kappa_s\) is singular
        in \(\HOD_t\), \(\kappa_t < \kappa_s\), contrary to the minimality of \(\kappa_s\).
        The proof of \cite[Theorem 2.8]{Goldberg2024StronglyCompactOD} 
        now shows that if \(\kappa_s \leq \delta < \xi\) and \(\delta\) is regular in \(\HOD_s\),
        then in \(N\), \(\cf(\delta) = |\delta|\). 
        Since \(\xi\) is a limit cardinal in \(N^+\),
        it follows that \(\xi\) is a limit of regular cardinals.
    \end{proof}
\end{thm}
Note that \cref{thm:dichotomy} \ref{item:supercompact} implies that one can force choice over
\(N\) by collapsing \(\kappa^\omega\); in this case,
one might say \(N\) is not truly \textit{Nairian}. This raises the question of whether
this alternative is even possible.
\cref{thm:dichotomy} \ref{item:supercompact} can be ruled out under fairly general hypotheses, 
including the central cases of minimal Nairian models and full Nairian models:
\begin{prp}\label{prp:hpc_no_sc}
    Assuming \(\AD_\mathbb R + \textnormal{HPC} + V = L(P(\mathbb R))\), \cref{thm:dichotomy} \ref{item:eventual_complete} holds.
    \begin{proof}
        Under these assumptions, 
        Steel's analysis of \(\HOD\) as an lbr hod pair \cite{Steel2022Comparison} shows that in \(\HOD\),
        there is no cardinal \(\kappa\) that is \(\kappa^+\)-supercompact.  
    \end{proof}
\end{prp}

\begin{prp}\label{prp:cardinal_no_sc}
    Assuming \(\xi\) is a cardinal in \(V\), \cref{thm:dichotomy} \ref{item:eventual_complete} holds.
    \begin{proof}
        By a theorem of Woodin \cite[Theorem 7.24]{Woodin2017UltimateL}, 
        assuming \(\AD^+ + V = L(P(\mathbb R))\), if \(\xi\) is a cardinal,
        then in \(\HOD\), no \(\kappa < \xi\) is \({<}\xi\)-supercompact.
    \end{proof}
\end{prp}

These propositions show that in the central cases,
\cref{thm:dichotomy} \ref{item:eventual_complete} holds. 
On the other hand, \cref{prp:not_vacuous} below 
shows that assuming the existence of very large cardinals in \(\HOD\) and 
a natural condition on the Sargsyan sequence, it is possible for
\cref{thm:dichotomy} \ref{item:supercompact} to hold in a Nairian model \(N = N_\xi\)
such that \(\xi\) is inaccessible in \(N^+\).

If \(X\) is ordinal definable, \(\Theta_0(X)\) denotes the
supremum of all ordinals \(\alpha\) such that there is an 
ordinal definable surjection from \(X\) onto \(\alpha\).
\begin{conj}\label{conj:ssh}
    Assume \(\AD^+ + V = L(P(\mathbb R))\). 
    Suppose \(\gamma < \Theta\) and the following hold:
    \begin{enumerate}[(1)]
        \item \(\gamma\) is a cardinal in \(\HOD\).
        \item No ordinal \(\kappa < \gamma\) is \({<}\gamma\)-supercompact in \(\HOD\).
        \item \(\cf(\gamma) = \omega\).
    \end{enumerate}
    Then \(\Theta_0(\gamma^\omega)\) is the least Woodin cardinal of \(\HOD\) above \(\gamma\).
\end{conj}
In \cref{prp:not_vacuous}, \textit{we drop the assumption that \(\xi\) is inaccessible
in \(N^+\)}, since our conclusion is simply that this must hold.
\begin{prp}\label{prp:not_vacuous}
    Assume \cref{conj:ssh}. If \(\xi\) is inaccessible in \(\HOD\)
    and \(\HOD\cap V_\xi\) satisfies that there is a supercompact cardinal,
    then \(\xi\) is inaccessible in \(N^+\).
    \begin{proof}
        Let \(\kappa\) be the least supercompact cardinal of \(\HOD\cap V_\xi\).
        We begin by computing \(\Theta_0(\kappa^\omega)\).
        Note that \(\kappa\) is a limit of cardinals \(\gamma\) satisfying
        the hypotheses of \cref{conj:ssh}. 
        Since \(\kappa\) is a limit of Woodins in \(\HOD\), 
        this implies that \(\Theta_0(\gamma^\omega) < \kappa\) for all \(\gamma < \kappa\). 
        Moreover by the \(\HOD\)-ultrafilter hypothesis (\cref{thm:hod_uf}), \(\cf(\kappa)\geq \omega_1\).
        Therefore \(\Theta_0(\kappa^\omega) = \kappa^{+\HOD}\).
        
        Let \(\mathbb P\in \HOD\) be the product Vop\v{e}nka forcing adding a sequence of elements of
        \(\kappa^\omega\) to \(\HOD\).
        By counting ordinal definable \(\infty\)-Borel codes, 
        one can show \(\HOD\vDash |\mathbb P| \leq 2^{<\theta}\) where \(\theta = \Theta_0(\kappa^\omega)\).
        In particular, \(|\mathbb P| < \xi\). 
        Since \(\HOD_{\kappa^\omega}\) is an inner model of a \(\mathbb P\)-generic extension of
        \(\HOD\), it follows that \(\xi\) remains inaccessible in \(\HOD_{\kappa^\omega}\).

        By \cref{prp:gamma_kappa_cov}, \(\HOD\) has the \((\aleph_1,\kappa)\)-cover property
        below \(\xi\). Therefore \(N^+ = \HOD_{\kappa^\omega}\),
        and so \(\xi\) remains inaccessible in \(N^+\), proving the proposition.
    \end{proof}
\end{prp}

We use the following theorem, abstracted from work of Adkisson--Ben-Neria \cite{AdkissonBenNeria2025},
to connect \cref{thm:dichotomy} to Gitik's theorem that ZF does not
prove the existence of a regular uncountable cardinal. 
\begin{thm}[Adkisson--Ben-Neria]\label{thm:abn}
    Suppose that there is a class function that assigns
    to each regular cardinal \(\delta\) a uniform measure \(U_\delta\in \meas(\delta)\) so that
    the function \(\delta \mapsto \kappa_{U_\delta}\) is monotone and unbounded. 
    Then there is a class symmetric extension
    satisfying \(\ZF\) in which every uncountable cardinal is singular.
    \begin{proof}
        We may assume there is a proper class of regular cardinals.
        Let \[C = \{\gamma \in \Ord : |\{\kappa_{U_\delta}: \delta\in \Reg\cap \gamma\}| = \gamma\}\]
        and let \(\langle \kappa_i : i \in \Ord\rangle\) enumerate the Beth fixed-points in \(C\) in increasing order.
        (An ordinal \(\kappa\) is a \textit{Beth fixed-point}
        if for all \(a\in V_\kappa\), \(\Theta(a) < \kappa\).)
        Let \(\vec U = \langle U_\delta: \delta\in \Reg\rangle\). 
        Let \(\mathbb P = \mathbb P_\text{fin}(\vec U)\) be Adkisson--Ben-Neria's class forcing 
        from \cite[Definition 56]{AdkissonBenNeria2025}.
        The proofs of \cite[Lemma 57 and Proposition 64]{AdkissonBenNeria2025},
        showing that Adkisson--Ben-Neria's symmetric submodel of \(V^\mathbb P\) 
        satisfies \(\ZF\) plus all uncountable cardinals are singular, 
        ostensibly require the existence of a rank Berkeley cardinal.
        In fact, their proofs go through in our context, with 
        the sequence of almost supercompact cardinals replaced by 
        the sequence \(\langle \kappa_i : i \in \Ord\rangle\) defined above. 
    \end{proof}
\end{thm}

\begin{cor}
    Assuming \(\textnormal{HPC}\) or that \(\xi\) is a \(V\)-cardinal, 
    \(N\) has a symmetric extension
    satisfying \(\ZF\) in which every uncountable cardinal is singular.\qed
\end{cor}

\begin{cor}\label{cor:gitik}
    The theory ``\textnormal{ZF + }every uncountable cardinal is singular'' 
    is consistent relative to a Woodin limit of Woodin cardinals.\qed
\end{cor}

\section{Comparing Nairian models}
We work in this section under the hypotheses:
\begin{itemize}
    \item \(\AD^+ + \AD_\mathbb R + V = L(P(\mathbb R))\).
    \item \(\xi < \Theta\) is strongly inaccessible in \(N^+ = N_\xi^+\).
\end{itemize}
Our next theorem is a Nairian generalization of a key consequence of determinacy.
Assuming \(\AD_\mathbb R\), Woodin \cite{WoodinUniqueness} showed that for every set \(X\) with 
\(|X| \leq^* |\mathbb R|\), there is a unique normal fine measure 
\(\mu_X\) on \(P_{\omega_1}(X)\). 

If \(X\in N\), 
then the restriction
of \(\mu_X\) to \(N\) belongs to \(N\) and witnesses that 
\(\omega_1\) is \(X\)-supercompact there.
In analogy with Woodin's theorem, 
we will prove that this is the 
\textit{unique} normal fine measure on \(P_{\omega_1}(X)\)
in \(N\).

The main idea of the proof is that the ``iteration maps'' from \cref{thm:ua}
can be lifted from their original domains, which are countable hulls of HOD, 
to act on the Nairian models associated to these hulls.

We need the following strengthening of \cref{thm:derived}.
\begin{thm}[Woodin]\label{thm:derived2}
    In the context of \cref{thm:derived}, for any
    \(\gamma<\Theta\), the model \(M\)
    and the embedding \(j\) can be chosen so that \(j\restriction\gamma\in M\).
\end{thm}
Note that \(j\restriction\gamma\in M\) implies that
\(M\cap P(\gamma)=\HOD\cap P(\gamma)\).

\begin{thm}\label{thm:unique}
    In \(N\), 
    for all \(X\), there is a unique normal fine measure on \(P_{\omega_1}(X)\).
\end{thm}
For the proof we need the following addendum to Spector's \cref{thm:spector}:
\begin{thm}[Spector]
    In the context of \cref{thm:spector}, if \(U\) is a normal fine measure on \(P_{\omega_1}(X)\), then \(k(\id_U) = j[X]\)
    and \(M\vDash |j[X]| \leq \aleph_0\).
\end{thm}
See the proof of \cite[Theorem 9]{Spector1988} and the comments in \cite[p.~597]{Spector1991}.

    \begin{proof}[Proof of \cref{thm:unique}]
        Suppose \(H\in \HOD\) is a countable transitive model 
        of \(\ZF\) admitting an elementary embedding \(\rho : H \to \mathscr H\)
        with \(\xi\in \ran(\rho)\), and let \(\bar \xi = \rho^{-1}(\xi)\). Applying
        \(\AD_\mathbb R\), let \(\gamma \in (\xi,\Theta)\) be least such that
        every \(s\in I_\xi\) is \({<}\gamma\)-generic over \(\HOD\),
        and let \(\bar \gamma = \rho^{-1}(\gamma)\).

        We will use the construction from \cref{thm:ua}.
        Let \(\tilde{M}\) denote the model \(M\) from 
        \cref{thm:derived2}, and let
        \(M = \rho^{-1}[\tilde{M}]\).
        Let \(D\) be a derived model of \(M\) at the 
        appropriate limit of Woodin cardinals and let
        \(i : H \to \mathscr H^D\) be the embedding
        corresponding to the one from \cref{thm:derived2}.
        We may assume that \(i\restriction \bar \gamma\in M\cap D\).
        In particular, \(M\cap P(\bar \gamma) = H\cap P(\bar \gamma)\).

        Suppose \(\bar{\mathcal N}\in \HOD\) is an outer model of
        \(H\) such that \(\rho\) is partial elementary from
        \((\bar{\mathcal N},H)\) to \((N^+,\mathscr H)\).
        Let \(\bar I = I_{\bar \xi}\cap \bar{\mathcal N}\).
        By our choice of \(\bar \gamma\), every
        \(s\in \bar I\) is \({<}\bar\gamma\)-generic over
        \(H\), and since \(M\cap P(\bar \gamma) = H\cap P(\bar \gamma)\),
        it follows that the same is true for \(M\).

        Let \(\mathbb R^* = \mathbb R\cap \tilde{M}\), and
        as in \cref{thm:ua},
        let \(\pi : M \to M^*\) be a Neeman-style \(\mathbb R^*\)-genericity
        iteration with critical point above \(\bar \gamma\). 
        Let \(D^*\) be the derived model of \(M^*\) computed in
        \(M^*(\mathbb R^*)\), and recall as in 
        \cref{thm:ua} that \(P(\mathbb R^*)\cap D^* \unlhd \Hom^{\tilde{M}}\). 
        
        Note that \(k = \pi \circ i\) is an elementary embedding
        from \(H\) to \(H^* = \mathscr H^{D^*}\).
        Moreover, since
        \(i\restriction \bar \xi\in M\cap D\) and \(\crit(\pi) > \bar \xi\),
        we have \(k\restriction \bar\xi = 
        \pi(i\restriction \bar\xi) \in D^*\).
        We claim that \(k\) lifts to an elementary embedding
        \[\hat{k} : (\bar{\mathcal N},H)\to (\mathcal N^*,H^*)\]
        where \(\mathcal N^* = (N_{k(\bar \xi)}^+)^{D^*}\).
        Namely, if \(a\in \bar{\mathcal N}\) is definable in 
        \((\bar{\mathcal N},H)\) from \(s\in \bar I\), set \(\hat{k}(a)\) 
        equal to the element of \(\mathcal N^*\) defined 
        in \((\mathcal N^*,H^*)\) from \(k[s]\) in the same way.
        Note here that \(k[s]\in D^*\) since \(k\restriction \bar\xi\in D^*\).
        The fact that \(\hat{k}\) is well-defined and elementary follows from
        the fact that \(\pi\circ i\) extends, for each \(s\in \bar I\),
        to an elementary embedding from \(H_s\) to \(\mathscr H^{D^*}_{k[s]}\),
        using \cref{thm:derived2} and the fact that \(\crit(\pi) > \bar\gamma\).

        Since it will be relevant in \cref{thm:ua_N},
        we note that the embedding \(\hat k\) is close to \(\bar{\mathcal N}\). Fix ordinals \(\delta\) and \(\alpha < k(\delta)\), and let
        \(U\) be the \(\bar{\mathcal N}\)-measure on \(\delta\) derived from \(\hat{k}\) using \(\alpha\).
        Let 
        \(\tilde{U}\) be the \(M\)-ultrafilter on \(i(\delta)\) derived from \(\pi\) using \(\alpha\). 
        Since \(\pi\) is close to \(M\), \(\tilde{U}\in M\). Moreover, since \(\crit(\pi) > \bar\gamma\), \(\kappa_{\tilde{U}} > \bar\gamma\),
        and therefore for each \(s\in \bar I\), Levy--Solovay implies
        \(\tilde{U}\) generates a measure \(\tilde{U}_s\) 
        in \(M[s] = M[i[s]]\).
        Finally, \[U = \{A\in P(\delta)\cap \bar{\mathcal N} : \exists s\in \bar I\, (A\in H[s]\wedge i_s(A)\in \tilde{U}_s)\}\]
        and this is definable over 
        \((\bar{\mathcal N},H)\) from the parameter 
        \(\tilde{U}\in M\subseteq H\).

        Now fix a finitely generated structure \(H\)
        with an elementary \(\rho : H \to \mathscr H\)
        such that \(\rho(\bar \xi) = \xi\) and \(\rho(\bar \Theta) = \Theta\).
        Construct \(\bar {V}\in \HOD\) by a type omission argument
        so that \(H = \mathscr H^{\bar {V}}\).
        Let \(\bar{\mathcal N} = (N_{\bar \xi}^+)^{\bar V}\).

        Suppose that \(X,\mathcal U\in \bar{\mathcal N}\) and \(\bar{\mathcal N}\)
        satisfies that \(\mathcal U\) is
        a normal measure on \(P_{\omega_1}(X)\). 
        We will show that \(\mathcal U\) is the measure derived from
        the embedding \(\hat{k} : \bar{\mathcal N}\to \mathcal N^*\)
        defined above using \(\hat{k}[X]\).

        Using Spector's ultrapower construction, form in \(\HOD\)
        an elementary \(j : \bar{\mathcal N}\to \mathcal P\)
        such that \(\mathcal U\) is the measure
        derived from \(j\) using \(j[X]\) and \(j\restriction (P(\Ord)\cap \bar{\mathcal N})\)
        agrees with the ultrapower embedding of \(\mathcal U\).
        Let \(H_{\mathcal P} = j(H)\). 
        By a standard realizability argument, there is an elementary embedding
        \(\rho_{\mathcal P} : H_{\mathcal P}\to \mathscr H\)
        such that \(\rho = \rho_{\mathcal P}\circ j\restriction H\).
        This embedding is automatically partial elementary from \(\mathcal P\) to \(N^+\) since the first-order theory of \(N^+\) with ordinal parameters uniformly reduces to that of \(\mathscr H\).

        Applying the construction above to \(\mathcal P\) instead of \(\bar{\mathcal N}\),
        we obtain \(\hat{k}_{\mathcal P} : \mathcal P \to \mathcal N^*_{\mathcal P}\).
        Using that \(H\) is finitely generated, one 
        can show as in \cref{thm:ua} that the relevant derived models are equal.
        Moreover,
        \(\hat{k}\restriction H = \hat{k}_{\mathcal P}\circ j\restriction H\) since \(H\) is finitely generated
        and these are close embeddings. 
        We can now conclude that \(\mathcal N^*_{\mathcal P} = \mathcal N^*\). 
        Moreover, it follows that
        \(\hat{k}\) and \(\hat{k}_{\mathcal P}\circ j\) agree on \(\bar I\), and hence on \(\bar{\mathcal N}\).
        In particular, \(\hat{k}_{\mathcal P}(j[X]) = \hat{k}_{\mathcal P}\circ j[X] = \hat{k}[X]\).
        Since \(\mathcal U\) is derived from \(j\) using \(j[X]\), we also have that
        \(\mathcal U\) is derived from \(\hat{k}_{\mathcal P}\circ j\) using \(\hat{k}_{\mathcal P}(j[X])\). 
        In other words, \(\mathcal U\) is the measure derived from \(\hat{k}\) using \(\hat{k}[X]\).
    \end{proof}
\cref{thm:unique} confirms a consequence of Sargsyan's 
quasi-club conjecture.
The \textit{quasi-club filter on \(P_{\omega_1}(X)\)} is
the filter generated by sets
of the form  \[\{\sigma \cap X : \sigma \preceq \mathcal M\text{ and }|\sigma| \leq \aleph_0\}\] 
where \(\mathcal M\) is a structure in a countable language whose
underlying set contains \(X\). 
The \textit{quasi-club conjecture} states that assuming 
\(\AD_{\mathbb R} + \Theta\text{ is regular} + \text{HPC} + V = L(P(\mathbb R))\), 
if \(N\) is a minimal Nairian 
model and \(X\in N\), then
the quasi-club filter on \(P_{\omega_1}(X)\) is an ultrafilter.
This conjecture probably cannot be proved as easily as \cref{thm:unique}.
For one thing, the conclusion of the quasi-club conjecture is false in the model of 
\cref{prp:not_vacuous}.

Next, we show that the Ultrapower Axiom holds in Nairian models.
See \cite{UAFromDet} for the statement of the Ultrapower Axiom in the context of ZF + DC;
note that it only applies to measures on ordinals.
\begin{thm}\label{thm:ua_N}
    The Ultrapower Axiom holds in \(N^+\).
    \begin{proof}
        Fix a finitely generated structure \(H\)
        with an elementary \(\rho : H \to \mathscr H\)
        such that \(\rho(\bar \xi) = \xi\) and \(\rho(\bar \Theta) = \Theta\).
        Construct \(\bar {V}\in \HOD\) by \cref{prp:hod-realization}
        so that \(H = \mathscr H^{\bar {V}}\).
        We will prove that \(\UA\) holds in \(\bar{\mathcal N} = (N_{\bar \xi}^+)^{\bar V}\).

        Suppose \(j_0 :\bar{\mathcal N}\to \mathcal P_0\) and 
        \(j_1 : \bar{\mathcal N}\to \mathcal P_1\) are Spector ultrapowers
        of \(\bar{\mathcal N}\) by \(\bar{\mathcal N}\)-internal measures 
        \(U_0\) and \(U_1\) on ordinals.
        The construction from \cref{thm:unique} produces close embeddings
        \(k_0 : \mathcal P_0\to M\) and \(k_1 : \mathcal P_1\to M\)
        such that \(k_0\circ j_{0} = k_1\circ j_{1}\).
        The arguments from \cite{UAFromDet} show this implies that \(\UA\)
        holds for \(U_0\) and \(U_1\).
    \end{proof}
\end{thm}

\subsection{UA and the filter extension property}
\begin{defn} 
    We say \(\delta\) is an \textit{NT-cardinal} if the following hold for some infinite \(\alpha\):
    \begin{itemize}
        \item  \(\delta = \Theta(V_\alpha) = \aleph(V_{\alpha+1})\).
        \item \(V_{\alpha+1}\) is the union of an increasing sequence  \((\Gamma_\beta)_{\beta < \delta}\) such that \(|\Gamma_\beta|\leq^* |V_\alpha|\).
        \item If \(M\) is transitive and \(M^{V_\alpha}\subseteq M\), there is \(X\preceq M\) with \(X^{V_\alpha}\subseteq X\) and \(|X|\leq^* |V_{\alpha+1}|\).
    \end{itemize}
\end{defn}
Blue--Larson--Sargsyan \cite{FailureSquare} show that minimal Nairian models satisfy the theory \(\textnormal{NT}_\text{base}\). If \(N\vDash\textnormal{NT}_\text{base}\), then for all ordinals \(\alpha\), if \(\cf(\alpha)\neq \omega\), \(\Theta(V_{\alpha+1})\) 
is an NT-cardinal.
\begin{thm}\label{thm:NT-filter-extension}
    Let \(\mathcal H = \HOD\cap V_\xi\) and work in \(N\). Suppose \(\lambda\) is an NT-cardinal or a limit of NT-cardinals such that \(\Theta(P(\rho)) \leq \lambda\) whenever \(\rho^+ < \lambda\).
    Then one of the following holds: 
    \begin{enumerate}[(1)]
        \item \label{item:filter_extension} Every \(\lambda\)-complete filter on an ordinal
    extends to a \(\lambda\)-complete ultrafilter.
        \item \(\mathcal H\) correctly computes all finite successors of \(\lambda^{++}\).
    \end{enumerate}
\end{thm}
For each \(x\), let \(\mathcal H_x = \HOD_x\cap V_\xi\) and \(\mathcal{O}_x = \OD_x\cap N\). 

\textit{For the rest of the section we work in \(N\) and fix
\(\lambda\) as in the theorem statement.}

Thus we have access to \(\ZF + \DC + \UA\), and 
by \cref{thm:ua_hod_s}, for all \(s\in \Ord^\omega\),
\(\mathcal H_s\vDash \UA\).

The proof of \cref{thm:NT-filter-extension} 
will also show that the failure of \ref{item:filter_extension} yields an \(s\in \Ord^\omega\) such that there is a supercompact cardinal in \(\mathcal H_s\cap V_\lambda\).
Moreover by the following lemma, \(\lambda\) is either inaccessible or a limit of inaccessibles in \(\mathcal H_s\):
\begin{lma}\label{lma:nt_inacc}
    If \(\delta\) is an NT-cardinal, then \(\delta\) is regular
    and \(\delta\) is inaccessible in any inner model of \(\ZFC\).
    \begin{proof}
        Fix \(\alpha\) such that \(\delta = \Theta(V_\alpha)\).
        By \(\DC\), \(\delta\) has uncountable cofinality.
        To see \(\delta\) is regular, note that combining the final two clauses of the definition of an NT-cardinal, one obtains that 
        for all sufficiently large ordinals \(\theta\), for any \(\gamma < \delta\) and \(a\in V_\theta\), there is some \(X\preceq V_{\theta+1}\) with \(\gamma\cup\{a\}\subseteq X\) and \(|X| \leq^* |V_\alpha|\). An argument similar to \cref{lma:limit-successor-regular} from \cref{thm:theta_meas} then implies \(\delta\) is regular.

        Now to see \(\delta\) is inaccessible in any inner model \(M\) of ZFC, we just need to show \(\delta\) is a strong limit cardinal in \(M\). For any \(\rho < \delta\), \(|\rho|\leq^* |V_\alpha|\), and so \(|P(\rho)|\leq |V_{\alpha+1}|\), which implies \(|P(\rho)\cap M| < \aleph(V_{\alpha+1}) = \delta\).
        Thus \((2^\rho)^M < |P(\rho)\cap M|^+\leq \delta\). 
    \end{proof}
\end{lma}
The following is a version of \cref{thm:main} for NT-cardinals.
Note that the proof does not use \(\Theta(V_\alpha) = \aleph(V_{\alpha+1})\), and hence it gives an alternate (though quite similar)
proof of \cref{thm:theta_meas}.
\begin{lma}\label{lma:hod_transfer}
    Suppose \(\delta\) is a regular cardinal
    and \(\kappa_\delta < \delta\).
    \begin{enumerate}[(1)]
        \item\label{item:strong_limit} If there is a strong limit cardinal \(\gamma\) such that \(\kappa_\delta < \gamma \leq \delta\leq \gamma^+\), then \(\mathcal H\) correctly computes all finite successors of \(\gamma\).
        \item\label{item:NT} If \(\delta\) is an NT-cardinal, then \(\mathcal H\) correctly computes all finite successors of \(\delta^{++}\).

    \end{enumerate}
    \begin{proof}
        Let \(\kappa = \kappa_\delta\). Applying \cref{cor:ua_completeness}, fix \(s\) such that \(\kappa = (\kappa_{\mathscr K_\delta})^{\mathcal H_s}\). By 
        \cref{prp:kappa_cov}, every \(S\subseteq \delta\)
        belongs to \(\mathcal H_{s,x}\) for some \(x\subseteq \kappa\).

        First assume \ref{item:strong_limit}. Then for all \(x\subseteq \kappa\), \(\mathcal H_{s,x}\) is a \(\gamma\)-cc generic extension of \(\mathcal H_s\).
        Thus as in \cref{thm:main},
        \(\mathcal H_s\) correctly computes all finite successors of \(\gamma\).

        By the definition of a strong limit cardinal, we have that \(\Theta(P(\kappa^+)) \leq \gamma\). Applying the Hamkins uniqueness theorem as in \cref{thm:main},
        every bounded subset of \(\gamma^{+\omega}\) belongs to
        a \(\gamma\)-cc extension of \(\mathcal H\), and
        therefore \(\mathcal H\) correctly computes all finite successors of \(\gamma\).

        Instead assume \ref{item:NT}, and fix \(\alpha\) such that
        \(\delta = \Theta(V_\alpha)\).
        We claim that in this case, we can pick \(s\) such that
        \(\mathcal H_{s,x}\) is a 
        \(\delta\)-cc generic extension of \(\mathcal H_s\) for each \(x\subseteq \kappa\).
        To achieve this, we must choose \(s\) so that in addition
        \(P(\kappa) = \bigcup_{\beta< \delta}S_\beta\)
        with \(|S_\beta|\leq^* |V_\alpha|\).
        Using that each \(S_\beta\) belongs to \(\mathcal O_s\), Bukovsky's criterion \cite{Bukovsky,FriedmanFuchinoSakai} 
        implies \(\mathcal H_{s,x}\) is a \(\delta\)-cc
        generic extension of \(\mathcal H_s\).

        Now as in \cref{thm:main}, we obtain that
        \(\mathcal H_s\) correctly computes all finite successors of \(\delta\) and \(\kappa\) is \({<}\delta^{+\omega}\)-supercompact in \(\mathcal H_t\) for all \(t\in \Ord^\omega\) such that \(\mathcal H_s\subseteq \mathcal H_t\). 

        We claim that \(\Theta(P(\delta)) \leq \delta^{++}\).
        Fix a surjection \(f : P(\delta)\to \eta\), and let us show that \(\eta < \delta^{++}\).
        First note that
        \(\Theta(P(\kappa))\leq \delta^+\)
        since \(P(\kappa) = \bigcup_{\beta< \delta}S_\beta\)
        where \(|S_\beta|\leq^* |V_\alpha|\).

        Moreover \(P(\delta) = \bigcup_{a\in P(\kappa)} Y_a\) 
        where \(Y_a = P(\delta)\cap \mathcal H_{s,a}\).
        Note that \(\mathcal H_{s,a}\vDash 2^\delta = \delta^+\): fix \(t\in \Ord^\omega\) such that
        \(\mathcal H_{s,a}\subseteq \mathcal H_t\); since
        \(\mathcal H_t\) satisfies \(\UA\) plus
        \(\kappa\) is \({<}\delta^{+\omega}\)-supercompact,
        we have \(2^\delta = \delta^+\) in \(\mathcal H_t\) by \cite{UAGCH}; since \(\mathcal H_{s,a}\) is cardinal correct in the interval \([\delta,\delta^{+\omega})\), we obtain \(2^\delta = \delta^+\) in \(\mathcal H_{s,a}\).
        Thus using the canonical well-order of \(\mathcal H_{s,a}\), we can fix bijections \(g_a : \delta^+\to Y_a\).
        Now if \(f : P(\delta)\to \eta\),
        then 
        \[\eta = \bigcup_{a\in P(\kappa)} f[Y_a] = \bigcup_{a\in P(\kappa)} (f\circ g_a)[\delta^+]\]
        It follows that \(|\eta|\leq^* |P(\kappa)\times \delta^+|\), and using \(\Theta(P(\kappa)) \leq \delta^+\), we obtain that \(\eta < \delta^{++}\).

        Applying the Hamkins uniqueness theorem as in \cref{thm:main}, it follows that
        every bounded subset of \(\delta^{+\omega}\) belongs to
        a \(\delta^{++}\)-cc extension of \(\mathcal H\), and
        therefore \(\mathcal H\) correctly computes all finite successors of \(\delta^{++}\).
    \end{proof}
\end{lma}

\begin{defn}
If \(U\) is an ultrafilter on \(X\) and \(P\) and \(Q\) are partitions of \(X\),
then \textit{\(P\) refines \(Q\) modulo \(U\)}, denoted \(P\leq_U Q\), 
if there is some \(A\in U\)
such that \(P\restriction A\) refines \(Q\restriction A\).

We say \(U\) is \textit{\(\gamma\)-decomposable} if there is a
partition \(P\) of \(X\) such that \(|P| = \gamma\) and 
whenever \(Q\leq_U P\),
\(|Q| \geq \gamma\). The \textit{decomposability spectrum} \(D(U)\) consists of all \(\gamma\)
such that \(U\) is \(\gamma\)-decomposable.
\end{defn}

\begin{lma}\label{lma:limit_spectrum}
    Suppose \(\gamma\) is a strong limit cardinal of countable cofinality or a limit of such cardinals. If there is a measure \(U\) on an ordinal
    with \(\sup(D(U)\cap \gamma) = \gamma\), then \(\mathcal H\) correctly computes all finite successors of \(\gamma\).
    \begin{proof}
        First assume \(\cf(\gamma) = \omega\). 
        Let \(\eta\) be the underlying ordinal of \(U\) and pick 
        \(\gamma_n\in D(U)\) increasing to \(\gamma\) and functions
        \(f_n : \eta\to \gamma_n\) witnessing \(\gamma_n\in D(U)\).
        Fix \(s\in \Ord^\omega\) such that \((f_n)_{n < \omega}\in \mathcal H_s\).
        Thus in \(\mathcal H_s\), 
        \(\cf(\gamma) = \omega\) and \(\gamma = \sup (D(\bar U)\cap \gamma)\)
        where \(\bar U = U\cap \mathcal H_s\). 
        The proof of \cite[Lemma 5.4]{GoldbergUniqueness} shows that
        in \(\HOD_s\), there is a countably complete uniform ultrafilter
        on \(\delta = (\gamma^+)^{ \mathcal H_s}\). 
        Therefore by \cite[Corollary 7.4.10]{UABook}
        some ordinal \(\kappa < \gamma\) is 
        \(\delta\)-supercompact in \( \mathcal H_s\). 
        The proof of \cite[Lemma 5.4]{GoldbergUniqueness} shows moreover that \(\delta = \gamma^+\): let
        \[X = \{\langle f_n(\xi) : {n < \omega}\rangle : \xi < \eta\}\] and note that \(X\in \mathcal H_s\), \(X\subseteq \gamma^\omega\), and \(|X| \geq \gamma^+\); therefore using Solovay's theorem on the SCH above strongly compact cardinals, \[\delta = (\gamma^+)^{ \mathcal H_s} = (\gamma^\omega)^{\mathcal H_s}\geq \gamma^+\] It follows that \(\delta\) is regular and
        \(\kappa_{\delta} < \delta\). By \cref{lma:hod_transfer},
        \(\mathcal H\) correctly computes all finite
        successors of \(\gamma\).

        We now consider the case where \(\gamma\) has uncountable cofinality. The previous case implies that there is an \(\omega\)-club \(C\) of cardinals \(\bar \gamma < \gamma\)
        such that \((\bar \gamma^+)^{\mathcal H} = \bar\gamma^+\).

        If \(\gamma\) is singular, then a Silver-style argument establishes
        that \((\gamma^{+})^{\mathcal H} = \gamma^+\).
        Namely, let \(\iota = \cf(\gamma)\), and let
        \(j : V \to M\) be the Spector ultrapower via \(W_\iota\).
        Note that \(\sup j[\gamma] = \gamma\) since \(\gamma\) is a strong limit, and hence \(\gamma\in j(C)\).
        Also \((\gamma^{+})^M\geq \gamma^+\): if \(\preceq\) is a well-order of \(\gamma\), then \({\preceq}^* = j({\preceq})\restriction \gamma\) is a well-order of \(\gamma\) that belongs to \(M\) with \(\ot({\preceq}^*)\) at least \(\ot({\preceq})\).
        This implies
        \[(\gamma^+)^{\mathcal H} \geq (\gamma^+)^{j(\mathcal H)} \geq (\gamma^+)^M \geq \gamma^+\]
        Thus \(\gamma^+\) is regular and we can apply \cref{lma:hod_transfer} to \(\delta = \gamma^+\);
        note that \(\kappa_{\gamma^+} < \gamma\) since 
        \(\kappa_{\gamma^+}\) is regular.

        If \(\gamma\) is regular, we can apply \cref{lma:hod_transfer}
        directly to \(\delta = \gamma\) once we show that
        \(\kappa_\gamma < \gamma\).
        Assume towards a contradiction that \(\kappa_{W_\gamma} = \gamma\). 
        Note that for a \(W_\gamma\)-large set of \(\bar \gamma\),
        \(\mathcal H\) correctly computes \(\bar\gamma^+\) and hence \(\mathcal H\) 
        satisfies that some \(\kappa < \bar \gamma\) is
        \(\bar \gamma^+\)-supercompact.
        Since \(W_\gamma\) is normal, 
        there is a fixed \(\kappa < \gamma\) that is \({<}\gamma\)-supercompact in \(\mathcal H\). By \cref{prp:gamma_kappa_cov}, it follows that \(\mathcal H\) has the \((\aleph_1,\kappa)\)-cover property at \(\gamma\),
        and by \cref{lma:ketonen},
        it follows that \((\kappa_{\mathscr K_\gamma})^{\mathcal H}\leq \kappa < \gamma\). By \cref{cor:ua_completeness}, this contradicts that \(\kappa_\gamma = \gamma\). 
    \end{proof}
\end{lma}

\begin{defn}
A \textit{finest partition of \(U\) below \(\lambda\)} is
a \(\leq_U\)-minimum element 
of the set of \({<}\lambda\)-sized partitions of the underlying set of \(U\).
\end{defn}
We prove the following result, based on a theorem of Silver \cite{SilverIndecomposable}.
\begin{lma}\label{lma:NT_spectrum}
    Suppose \(\delta\) is an NT-cardinal, \(U\) is a measure on an ordinal, and one of the following holds: 
    \begin{itemize}
        \item \(\sup (D(U)\cap \theta) < \delta\) where \(\theta = \aleph(P(\delta))\).
        \item \(\mathcal H\) does not correctly compute all finite successors of \(\delta^{++}\).
    \end{itemize}
    Then \(U\) has a finest partition below \(\delta\).
    \begin{proof}
        Fix \(\alpha\) such that \(\delta = \Theta(V_\alpha)\). 
        Fix \(\theta\) sufficiently large and take a \(V_\alpha\)-closed \(X\preceq V_{\theta+1}\) with \(|X| \leq^* |V_{\alpha+1}|\) containing the finitely many relevant parameters. Write 
        \[X = \bigcup_{\beta < \delta} X_\beta\] where \(|X_\beta| \leq^* |V_\alpha|\). Note that each \(X_\beta\in X\) by \(V_\alpha\)-closure. 
        Let \(\eta\) be the underlying ordinal of \(U\) and for \(\beta < \delta\), define a partition \(P_\beta\) of \(\eta\) by putting \(\xi_0,\xi_1 < \eta\) into the same cell if and only if \(D(X_\beta,\xi_0) = D(X_\beta,\xi_1)\) where \[D(X_\beta,\xi) = \{A\in P(\eta)\cap X_\beta: \xi\in A\}\] 
        Then \(P_\beta\in X\) since \(X_\beta\in X\).
        Moreover, we claim that \(|P_\beta| < \delta\). 
        First, note that \(P_\beta\) injects into \(P(X_\beta)\):
        map a cell \(B\in P_\beta\) to the common value
        of \(D(X_\beta,\xi)\) for all \(\xi\in B\).
        Since \(P_\beta\) is well-orderable, 
        \[|P_\beta| < \aleph(P(X_\beta)) \leq \aleph(V_{\alpha+1}) = \delta\]

        To show \(U\) has a finest partition, it suffices to show that this sequence stabilizes modulo \(U\), or in other words that there is some \(\beta < \delta\) such that for all \(\beta'\geq \beta\), \(P_\beta \equiv P_{\beta'}\). 
        Then since \(P_\beta\in X\), \(P_\beta\) is a finest partition of \(U\) below \(\delta\).

        Assume first that \(\sup(D(U)\cap \theta) < \delta\) where \(\theta = \aleph(P(\delta))\). Let \(Q\) be the least common refinement of \((P_\beta)_{\beta < \delta}\). Note that \(|Q|\leq |P(\delta)|\). Since \(Q\) is well-orderable, 
        \(|Q| < \aleph(P(\delta))\). Since \(\sup(D(U)\cap \theta) < \delta\), there is some \(A\in U\) such that \(|Q\restriction A| < \delta\).
        Thus there is then some \(\beta < \delta\) such that for all \(\beta'\geq \beta\), \(P_\beta\restriction A = P_{\beta'} \restriction A\). 
        Since \(P_\beta\in X\), \(P_\beta\) is a finest partition of \(U\) below \(\delta\).

        We finish by proving \(U\) has a finest partition below \(\delta\) in the case that \(\mathcal H\) does not correctly compute all finite successors of \(\delta^{++}\).
        Take \(s\in \Ord^\omega\) such that
        \(\langle P_\beta : \beta < \delta\rangle\in \mathcal H_s\).
        Let \(\bar U = U\cap \mathcal H_s\) and let \(\mu = \sup (D(\bar U)\cap \delta)\). If \(\mu < \delta\), then since \(\delta\) is inaccessible in \(\mathcal H_s\), the sequence \((P_\beta)_{\beta < \delta}\) stabilizes modulo \(U\): there cannot be a strictly Rudin--Keisler increasing \(\delta\)-sequence of measures on \(\mu\).

        So assume towards a contradiction that
        \(\mu = \delta\).
        Then some \(\kappa < \delta\) is \(\delta\)-supercompact 
        in \(\mathcal H_s\). 
        It follows from \cref{cor:ua_completeness} that \(\kappa_\delta < \delta\),
        and so by \cref{lma:hod_transfer}, \(\mathcal H\) correctly
        computes all finite successors of \(\delta^{++}\).
    \end{proof}
\end{lma}

\begin{defn}
    Suppose \(D\) and \(U\) are measures and \(X\) is the underlying set of \(D\).
    Then \[s_U(D) =  \{A\in j_U(P(X)) : j_U^{-1}[A] \in D\}\]
\end{defn}
In the ZFC context, \(s_U(D)\in M_U\) if and only if \(j_D\restriction M_U\) is definable over \(M_U\). Though this does not go through assuming just ZF + DC, a certain amount of the UA-theory of \(s_U(D)\) works out:
\begin{lma}[\(\ZF + \DC + \UA\)]\label{lma:inacc_to_internal}
    Suppose \(D\) and \(U\) are measures on ordinals,
    and \(U\) is not the
    \(D\)-limit of a sequence of measures Ketonen below \(U\).
    Then \(s_U(D)\in M_U\).
\end{lma}
\begin{proof}[Sketch]
    This is a choiceless analog of \cite[Lemma 5.5.15]{UABook}.
    The details are contained in the proofs of \cite[Lemma 3.19]{EvenOrdinals} and
    \cite[Theorem 3.8]{UAFromDet}.
\end{proof}

\begin{proof}[Proof of \cref{thm:NT-filter-extension}]
    Assume \(\mathcal H\) does not correctly compute all finite
    successors of \(\lambda^{++}\). 
    Let \(F\) be a \(\lambda\)-complete filter on an ordinal
    \(\eta\). Let \(U\) be the Ketonen least measure on \(\eta\)
    extending \(F\). We claim that \(U\) is \(\lambda\)-complete.
    
    We first claim that \(U\) has a finest partition below \(\lambda\).
    If \(\lambda\) is an NT-cardinal, this is immediate from \cref{lma:NT_spectrum}. If \(\lambda\) is a limit of NT-cardinals,
    we first use \cref{lma:limit_spectrum} to conclude that
    \(\sup(D(U)\cap \lambda) < \lambda\). Fix an NT-cardinal
    \(\delta < \lambda\) above this supremum. Note now that
    \(\sup (D(U)\cap \theta) < \delta\) where \(\theta = \aleph(P(\delta))\).
    Therefore \(U\) has a finest partition below \(\delta\),
    and it follows easily that \(U\) has a finest partition below \(\lambda\).

    Next, we show that if \(D\) is a measure on a cardinal \(\gamma < \lambda\),
    then \(s_U(D) \in M_U\).
    In fact, if \(\vec U = \langle U_\xi : \xi < \gamma\rangle\) 
    is a sequence of measures on \(\eta\) and
    \(U_\xi\ke U\) for all \(\xi < \gamma\), we show that 
    \(U^* = D\ulim_{\xi < \gamma} U_\xi\) is not \(U\), and then cite \cref{thm:ua_N} and \cref{lma:inacc_to_internal}.
    Note that by the minimality of \(U\),
    \(F\setminus U_\xi\neq \emptyset\) for each \(\xi < \gamma\).
    We will find a single \(S\in F\setminus \bigcup_{\xi < \gamma} U_\xi\).

    For this, fix an NT-cardinal \(\delta = \Theta(V_\alpha)\) with
    \(\gamma < \delta \leq \lambda\). Take a sufficiently large ordinal \(\theta\), and find an elementary substructure 
    \(X\preceq V_{\theta+1}\) with \(|X| \leq^* |V_{\alpha+1}|\),
    containing \(V_{\alpha+1}\cup \{F,\vec U\}\). 
    Write \(X\) as a continuous increasing union of sets \((X_\beta)_{\beta < \delta}\) 
    where \(|X_\beta| \leq^* |V_\alpha|\). Fix \(\beta < \delta\) such that \(\gamma\cup \{F,\vec U\}\subseteq X_\beta\) and 
    \(X_\beta \preceq V_{\theta+1}\). Then for each \(\xi < \gamma\),
    \((F\setminus U_\xi)\cap X_\beta \neq \emptyset\).
    Let \(S = \bigcap (F\cap X_\beta)\). Clearly
    \(S\notin U_\xi\), but also \(S\in F\) since \(F\)
    is closed under \(V_\alpha\)-indexed intersections
    by \cite[Lemma 3.5]{Goldberg2024MeasurableChoiceless},
    using the fact that \(\aleph(V_{\alpha+1}) = \delta\)
    and \(F\) is \(\delta\)-complete.

    Thus
    \(s_U(D) \in M_U\) for every \(D\in \bigcup_{\gamma < \lambda}\meas(\gamma)\). 
    Let \(Q\) be a finest partition of \(U\) below \(\lambda\),
    with \(|Q|\) as small as possible. Let \(\gamma = |Q|\), and fix an NT-cardinal \(\delta\leq \lambda\) with \(\gamma < \delta\).

    Let \(D\) be the projection of \(U\) to a measure on \(Q\):
    \[D = \{A\subseteq Q: \bigcup A\in U\}\]
    We view \(D\) as a measure on \(\gamma\);
    then \(s_U(D) \in M_U\). Moreover by the minimality of \(|Q|\),
    \(D\) is uniform.

    To show \(U\) is \(\lambda\)-complete, it suffices to show that \(D\) is principal, so assume towards a contradiction that \(D\) is nonprincipal.
    Then since \(D\) is uniform on \(\gamma\) and \(D\RK U\), 
    \(j_U\) is discontinuous at \(\gamma\).
    Thus \(s_U(D)\) concentrates on
    \(\sup j_U[\gamma] < j_U(\gamma)\).
    
    Let \(s_U(D) = [f]_U\), where \(f : \eta \to \meas(\gamma)\)
    and each \(f(\xi)\) concentrates on some \(\nu < \gamma\).
    We claim that \(|\ran(f)| < \lambda\).
    If \(\lambda\) is a limit of NT-cardinals, this is
    easy since for all \(\gamma < \lambda\),
    fixing \(\alpha\) such that \(\gamma\leq \Theta(V_\alpha) < \lambda\), we have
    \(|\ran(f)| \leq |\meas(\gamma)| \leq \Theta(V_{\alpha+10}) < \lambda\). The more delicate case occurs when \(\lambda\) is an NT-cardinal.
    Then we claim that for each \(\nu < \gamma\), 
    \(|\meas(\nu)| < \lambda\). 
    Note here that
    \(\nu^+ < \lambda\), so by assumption \(\Theta(P(\nu))\leq \lambda\). 
    Fix \(\alpha\) such that \(\lambda = \Theta(V_\alpha)\).
    By the definition of an NT-cardinal,
    we can write \(P(\nu)\) as an increasing union \(\bigcup_{\beta < \lambda} Y_\beta\)
    where \(|Y_\beta|\leq^* |V_\alpha|\). 
    Since \(\Theta(P(\nu)) \leq \lambda\), we must have
    \(P(\nu)\subseteq Y_\beta\) for some \(\beta < \lambda\),
    which means that \(|P(\nu)|\leq^* |V_\alpha|\).
    Then \(|\meas(\nu)|\leq |V_{\alpha+1}|\), and thus
    \(|\meas(\nu)| < \aleph(V_{\alpha+1}) = \lambda\).
    Since \(\lambda\) is regular by \cref{lma:nt_inacc}, 
    it follows that the set of measures on \(\gamma\) that concentrate on an ordinal \(\nu < \gamma\) has cardinality
    less than \(\lambda\). Hence \(|\ran(f)| < \lambda\), as claimed.

    Since \(|\ran(f)| < \lambda\) and \(Q\) is a finest partition of \(U\) below \(\lambda\), letting \(k : M_D\to M_U\) be the factor embedding, we have \(s_U(D) = [f]_U\in \ran(k)\).
    It follows that \[s_D(D) = k^{-1}[s_U(D)]\in M_D\]
    Since \(D\) is nonprincipal, \(s_D(D)\in M_D\) yields a
    strict Lipschitz reduction from \(D\) to itself:
    let \(\langle D_\alpha : \alpha < \gamma\rangle\) represent
    \(s_D(D)\), and note that \(\alpha\in D_\alpha\) for all \(\alpha < \gamma\)
    since \(D\) is uniform and \(s_D(D)\) concentrates on \(\sup j_D[\gamma]\);
    we have \(S\in D\) if and only if \(\{\alpha < \gamma : S\cap \alpha\in D_\alpha\}\in D\).
    This contradicts
    \cite[Corollary 3.4.29]{UABook}.

    Thus we have shown that \(D\) is principal, and hence \(U\) is \(\lambda\)-complete.
\end{proof}
For our final theorem, we assume \(\text{HPC}\) and that \(N\) is a minimal Nairian model. 
\begin{thm}\label{thm:filter_extension}
    In \(N\), suppose \(\alpha\geq \omega\) is not of the form \(\beta+1\) where \(\beta > \omega\) and \(\cf(\beta) = \omega\).
    Let \(\delta = \Theta(V_\alpha)\). Then every
    \(\delta\)-complete filter on an ordinal extends to a
    \(\delta\)-complete ultrafilter.
    \begin{proof}
        If \(\alpha = \omega\), then \(\Theta(V_\alpha) =\omega_1\), and we obtain the filter extension property from Kunen's theorem.
        If \(\alpha\) is a successor ordinal, then 
        \(\text{NT}_\text{base}\) combined with \cref{cor:no-long-rank-sequences} yields that \(\delta\) is an NT-cardinal, so we can apply \cref{thm:NT-filter-extension}. 
        If \(\alpha\) is a limit ordinal,
        then by \cref{lma:limit-successor-regular} of \cref{thm:theta_meas}, \(\Theta(V_\alpha) = \lambda^+\)
        where \(\lambda = \sup_{\beta < \alpha} \Theta(V_{\beta})\).
        Then \(\lambda\) is a limit of NT-cardinals, so by \cref{thm:NT-filter-extension}, every \(\lambda\)-complete filter on an ordinal extends to a \(\lambda\)-complete ultrafilter.
        By minimality, \(\lambda\) is singular, and therefore
        \(\lambda\)-complete filters are \(\lambda^+\)-complete. This yields the conclusion.
    \end{proof}
\end{thm}

\paragraph{AI statement.}
The main results of this paper were proved by the authors in summer 2026 without the use of AI. 
Two notable AI contributions came later.
First, the comparison argument in \cref{thm:unique} was loosely inspired by the following comment of GPT-5.6 Sol, which appeared on a much longer list of less helpful 
research directions: ``One should
compare all the \(\mathrm{HOD}_s\)'s generated by
\(s\in\bigcup_{\beta<\Lambda}\beta^\omega\), together with their
transition maps.'' Second, the authors had originally proved
\cref{cor:no-long-rank-sequences} using the HOD analysis,
but GPT-6 Astra suggested using Jech's
\cref{lma:small-completely-generated-algebra}
to complete the softer proof given here. 
The paper is human-written, but LLMs were consulted during its preparation.
The authors used GPT-5.6 Sol and later GPT-6 Astra to check arguments, find typos, uniformize notation, and compile the bibliography.

\paragraph{Acknowledgements.}

This work was partially supported by the Simons Foundation grant (award no. SFI-MPS-T-Institutes-00010825) and from State Treasury funds as part of a task commissioned by the Minister of Science and Higher Education under the project “Organization of the Simons Semesters at the Banach Center - New Energies in 2026-2028” (agreement no. MNiSW/2025/DAP/491).
\begin{center}
\includegraphics[width=0.9\linewidth]{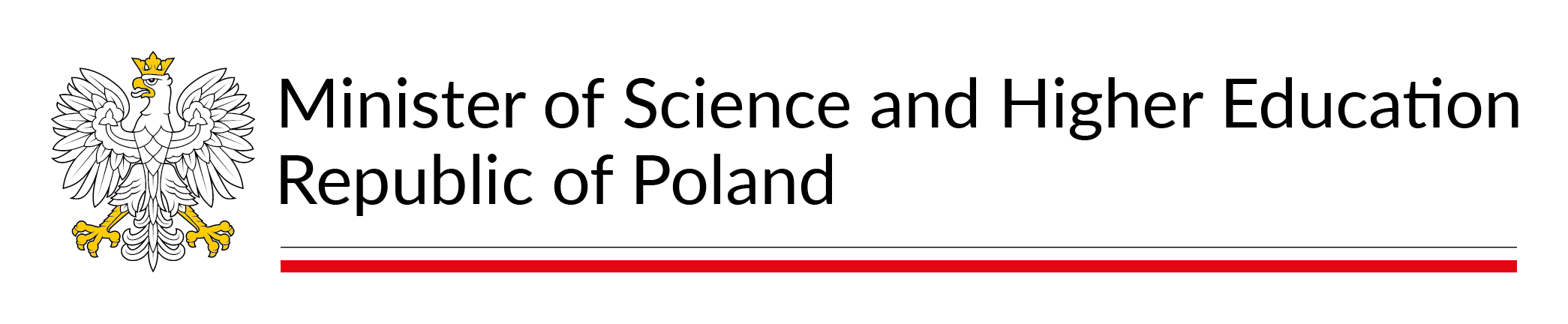}
\par\medskip
\includegraphics[width=0.9\linewidth]{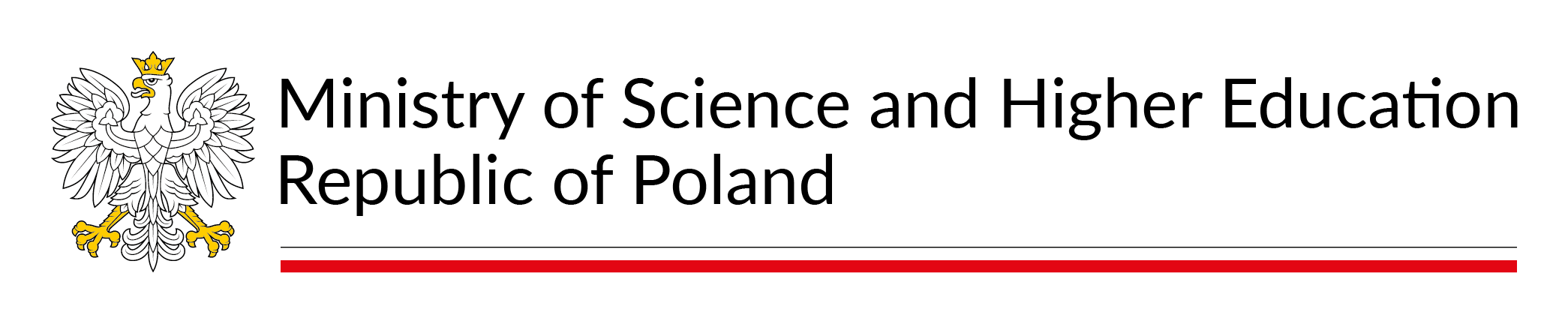}
\end{center}

\end{document}